\documentclass[11pt]{article}
\usepackage{amsfonts,amsmath, amssymb,latexsym}
\usepackage[OT2,OT1]{fontenc}
\usepackage[all,cmtip]{xy}
\def\SH{\mbox{\fontencoding{OT2}\selectfont\char88}}
\def\Z{{\mathbb Z}}

\def\SL{{\rm SL}}
\def\GL{{\rm GL}}

\def\Stab{{\rm Stab}}

\def\SO{{\rm SO}}
\def\P{{\mathbb P}}

\def\Aut{{\rm Aut}}
\def\irr{{\rm irr}}

\def\Inv{{\rm Inv}}
\def\red{{\rm red}}
\def\Vol{{\rm Vol}}
\def\R{{\mathbb R}}
\def\F{{\mathbb F}}
\def\FF{{\mathcal F}}

\def\Q{{\mathbb Q}}
\def\pv{{V_\Z^{+}}}
\def\nv{{V_\Z^{-}}}
\def\pnv{{V_\Z^{\pm}}}
\def\pr{{V_\R^{+}}}
\def\nr{{V_\R^{-}}}
\def\pnr{{V_\R^{\pm}}}

\def\prg{{R^{+}}}
\def\nrg{{R^{-}}}
\def\pnrg{{R^{\pm}}}
\def\H{{\mathcal H}}

\def\J{{\mathcal J}}
\def\NN{{\mathcal N}}
\def\ZZ{{\mathcal Z}}
\def\C{{\mathcal C}}

\def\W{{\mathcal W}}

\def\Z{{\mathbb Z}}
\def\P{{\mathbb P}}
\def\F{{\mathbb F}}
\def\Q{{\mathbb Q}}
\def\C{{\mathbb C}}

\def\H{{\mathcal H}}

\def\SS{{\mathcal S}}

\def\var{{\textrm{Var}}}
\newtheorem{theorem}{Theorem}

\newtheorem{lemma}[theorem]{Lemma}
\newtheorem{proposition}[theorem]{Proposition}
\newenvironment{proof}{\noindent {\bf Proof:}}{$\Box$ \vspace{2 ex}}
\newtheorem{rem}[theorem]{Remark}

\title{The average size of the $5$-Selmer group of elliptic curves is
  $6$, and the average rank is less than $1$}

\author{Manjul Bhargava and Arul Shankar}

\begin{document}
\maketitle
\section{Introduction}

The purpose of this article is to show that the average rank of 
elliptic curves over $\Q$, when ordered by height, is less than $1$
(in fact, less than $.885$). As a consequence of our methods, we also
prove that at least four fifths of all elliptic curves over $\Q$ have
rank either 0 or 1; furthermore, at least one fifth of all elliptic
curves in fact have rank 0.  The primary ingredient in the proofs of
these theorems is a determination of the average size of the 5-Selmer
group of elliptic curves over $\Q$; we prove that this average size is 6.
Another key ingredient is a new lower bound on the equidistribution of
root numbers of elliptic curves; we prove that there is a
family of elliptic curves over $\Q$ having density at least $55\%$ 
for which the root number is
equidistributed.

We now describe these results in more detail.  
Recall that any elliptic curve $E$ over $\Q$ is isomorphic to one of the form
$E_{A,B}:y^2=x^3+Ax+B$. If, for all primes $p$, we further assume that
$p^6\nmid B$ whenever $p^4\mid A$, then this expression is unique. 
We define the (naive)
height of $E_{A,B}$ by $H(E_{A,B}):=\max\{4|A^3|,27B^2\}$.
In previous work (\cite{BS}, \cite{TC}, and \cite{foursel}), we showed
that when elliptic curves over $\Q$ are ordered by height, the average sizes
of their $2$-, $3$-, and $4$-Selmer groups are given by $3$, $4$, and $7$,
respectively. These results, and their proofs, led us to
conjecture in \cite{foursel} 
that for all~$n$, the average size of the $n$-Selmer
group of elliptic curves over $\Q$, when ordered by height, is the sum of the
divisors of $n$.

In this paper, we prove the following theorem which confirms the 
conjecture when~$n=5$:

\begin{theorem}\label{mainellip}
When elliptic curves $E/\Q$ are ordered by height,
the average size of the $5$-Selmer group $S_{5}(E)$ is equal to $6$.
\end{theorem}
 We note that Theorem~\ref{mainellip} also confirms a case of the
 Poonen--Rains heuristics~\cite[Conjecture~1.1(b)]{PR}, 
 which predict in particular that for any
 prime number $p$, the average
size of the $p$-Selmer group of elliptic curves is $p+1$. 

We actually prove a stronger version of Theorem~\ref{mainellip},
where we determine the average size of the $5$-Selmer group over 
elliptic curves whose defining coefficients satisfy any finite set of
congruence conditions:

\begin{theorem}\label{ellipcong}
  When elliptic curves $E:y^2=x^3+Ax+B$ over $\Q$, in any family
  defined by finitely many congruence conditions on the coefficients
  $A$ and $B$, are ordered by height, the average size of the
  $5$-Selmer group $S_5(E)$ is~$6$.
\end{theorem}
Thus the average size of the 5-Selmer groups of elliptic curves in any
congruence family is independent of the family. 

We use Theorem~\ref{ellipcong},
together with some further ingredients to be described below, 
to obtain a number of results on the 
distribution of ranks of elliptic curves.  
The rank distribution conjecture, due to Goldfeld~\cite{G1} and
Katz--Sarnak~\cite{KS} (see also \cite{BMSW} for a beautiful survey,
where it is termed the ``minimalist conjecture''), states that the
average rank of elliptic curves should be 1/2, with 50\% of curves
having rank 0 and 50\% having rank 1.  However, prior to the work
\cite{BS} giving the average size of the 2-Selmer group of elliptic curves, 
it was not known unconditionally that the average rank of
elliptic curves is even finite. (Conditional on GRH and BSD, a finite
upper bound of 2.3 on the average rank was demonstrated by
Brumer~\cite{Brumer}; improved conditional
upper bounds of 2.0 and 1.79 were given by Heath-Brown~\cite{HB} and more recently by
Young~\cite{Young}, respectively.) 

In this article, we give the first proof that the average rank of
elliptic curves over $\Q$ is less than 1:

\begin{theorem}\label{thavgrank}
  When elliptic curves over $\Q$ are ordered by height, their
  average rank is $<.885$.
\end{theorem}

The rank distribution/minimalist conjecture predicts
that elliptic curves over $\Q$ should tend to have rank either 0 or 1.  We prove that
this is the case for the vast majority of elliptic curves over $\Q$:

\begin{theorem}\label{rank01}
When elliptic curves over $\Q$ are ordered by height, a density of 
at least $83.75\%$ have rank $0$ or $1$.
\end{theorem}
In~\cite{TC}, we showed that a positive proportion of elliptic curves
have rank 0; however, the proportion that we demonstrated there was quite
small.  As a consequence of our methods here, we are able to deduce
that a fairly significant proportion of elliptic curves have rank 0:

\begin{theorem}\label{rank0}
When elliptic curves over $\Q$ are ordered by height, a density
of at
least $20.62\%$ have rank $0$. 
\end{theorem}
If the Tate--Shafarevich groups of elliptic curves are finite, then
our methods also demonstrate that a proportion of at least $26.12\%$ of elliptic
curves have rank 1.

We now describe some of the methods behind the proofs of
Theorems~\ref{mainellip}--\ref{ellipcong} and \ref{thavgrank}--\ref{rank0}.  The key algebraic ingredient in
proving Theorems \ref{mainellip} and \ref{ellipcong} is a parametrization of elements  of
the 5-Selmer group of an elliptic curve.  Recall that an
element in the $5$-Selmer group of an elliptic curve $E$ may be
viewed as a locally soluble $5$-covering of $E$. Given any integer
$n\geq 1$,
an {\it $n$-covering} of
an elliptic curve $E/\Q$ is a genus one curve $C/\Q$ equipped with maps
$\phi:C\to E$ and
$\theta:C\to E$, where $\phi$ is an isomorphism defined over $\C$ and
$\theta$ is a degree $n^2$ map defined over $\Q$, such that the following
diagram commutes:
$$\xymatrix{E \ar[r]^{[n]} &E\\C\ar[u]^\phi\ar[ur]_\theta}$$
An $n$-covering $C$ of $E$ is said to be {\it soluble} if it has a
rational point, and {\it locally soluble} if it has a point over every
completion of $\Q$. 

Cassels \cite{Cassels} showed that any locally soluble $n$-covering of
$E$ admits a rational divisor of degree $n$, yielding a map $C\to
\P^{n-1}$ (which gives an embedding once $n\geq 3$).  In the case $n=2$,
we obtain double cover $C\to\P^1$ ramified at 4 points, and thus we may
describe 2-coverings $C$ of elliptic curves $E$ over $\Q$ via binary
quartic forms over $\Q$; this perspective on 2-coverings, which played
an important role in the original rank computations of Birch and
Swinnerton-Dyer~\cite{BSD}, was also the key to the work in \cite{BS}.  In
the cases $n=3$ and $n=4$, one obtains a genus one curve $C$ embedded 
in $\P^2$ or $\P^3$, respectively, which may be described as the locus of zeros in
$\P^2$ of a ternary cubic form or by a complete intersection of two
quadrics in $\P^3$, respectively (see, e.g., \cite{MMM} for an
excellent treatment over general fields).  
These geometric descriptions of
genus one normal curves of degrees 3 and 4 indeed played key roles in the
works on 3- and 4-Selmer groups in \cite{TC} and \cite{foursel}.  

In the case $n=5$, however, a genus one curve in $\P^4$ is {\it not} a
complete intersection.  Nevertheless, a genus one curve in $\P^4$ may
be expressed---essentially uniquely---as an intersection of the five 
quadrics defined by the
$4\times 4$ sub-Pfaffians of a $5\times 5$ skew-symmetric matrix of linear
forms!  Conversely, given a generic $5\times 5$ matrix of linear forms in
five variables, the $4\times 4$ sub-Pfaffians yield five quadrics,
whose intersection gives a genus one curve in $\P^4$.  These facts are
essentially classical over an algebraically closed field; over a
general field, they follow from the seminal 
Buchsbaum--Eisenbud structure theorem (see \cite{BE1}, \cite{BE2}). 

The theory over arithmetic fields, such as $\Q$ and $\Q_p$, has been
subsequently developed in a beautiful series of papers by
Fisher~\cite{Fisher1,Fisher3,Fisher15,Fishermin,Fisher2}.  In particular, Fisher
relates the invariant theory of the representation
$V=5\otimes\wedge^2(5)$ of the group $G'=\SL(5)\times \SL(5)$ to
explicit formulae for 5-coverings of elliptic curves.  The
representation of $G'$ on $V$ is classically known to have a remarkable
invariant theory; it is one of the first ``exotic'' examples arising
in Vinberg's theory of $\theta$-groups~\cite{popovvinberg}, and arises
in the classification of coregular spaces~\cite{littelmann}, i.e.,
representations having a free ring of invariants.  The invariant ring
of the representation of $G'_\C$ on $V_\C$ is freely generated by two
invariants, which we denote by $I$ and $J$, and which are integral
polynomials in the entries of~$V$.  Fisher shows that every locally
soluble 5-covering of an elliptic curve $E_{A,B}$ can be represented
by a genus one curve in $\P^4$ corresponding to an element $v\in
V_\Q$, where the $I$ and $J$ invariants of $v$ agree with $A$ and
$B$, respectively (up to bounded powers of 2 and 3); even more
remarkably, this element $v\in V_\Q$ can in fact always be taken in
$V_\Z$ (\cite[Theorem~2.1]{Fishermin}).

To obtain an exact parametrization of $5$-coverings of elliptic
curves, it is necessary to use the action of a slightly different group $G$, as
defined in \eqref{eqg}; see \cite[\S4.4]{BhHo} for a detailed
explanation.  
It then follows that 5-Selmer elements of elliptic curves
$E_{A,B}$ having bounded height can be represented by certain
$G_\Z$-orbits on $V_\Z$ having bounded invariants $I$ and $J$. 
Therefore, to count the total number of 5-Selmer elements of elliptic
curves of bounded height, it suffices to count the number of
corresponding $G_\Z$-orbits on $V_\Z$ having bounded invariants $I$
and $J$.  
To carry out such a count, we adapt the methods of
\cite{dodpf} and \cite{BS}.  Specifically, we construct fundamental
domains for the action of $G_\Z$ on $V_\R$, and then count lattice
points having bounded invariants in these domains.  As usual, the
difficulties lie in the (numerous) cusps of such a fundamental domain.
We divide the fundamental domain into two parts, the ``main body'' and
the ``cuspidal region''.  We show that in the main body, only a
negligible number of lattice points correspond to identity 5-Selmer
elements.  Meanwhile, in the cuspidal region, we show that a
negligible number of points correspond to non-identity 5-Selmer
elements!  The latter is proven by partitioning the cuspidal region
into thousands of subregions, on each of which the argument of
\cite{dodpf} is then applied.  If we actually wrote out the proof for
each such subregion as in \cite{dodpf}, it would take hundreds of
pages!  Thus we introduce a new, more uniform method of certifying
that the argument works on each of these subregions, allowing us to
present a fairly short proof that is checkable by hand; see Section~3.
We expect that this method will also be useful in other contexts in
handling geometry-of-numbers difficulties in complex cuspidal regions.

We conclude from the above results that, in order to count the number
of non-identity 5-Selmer elements of elliptic curves having bounded
height, it suffices to count lattice points in the main body of the
fundamental region satisfying suitable congruence conditions,
\pagebreak so that
we are counting each non-identity 5-Selmer element exactly once.  This
is also accomplished via geometry-of-numbers arguments, together with
a suitable sieve using the results of \cite{geosieve}.  The sieve
reveals that the average number of non-identity elements in the
5-Selmer group of elliptic curves is the Tamagawa number $(=5)$ of the
group $G$.  We conclude that the average number of elements in the
5-Selmer groups of elliptic curves is $5+1=6$, proving
Theorem~\ref{mainellip}.  An analogous argument, and the latter sieve,
also allows us to prove Theorem~\ref{ellipcong}.

We now describe how Theorem \ref{thavgrank} is
deduced (the deduction of Theorems~\ref{rank01} and \ref{rank0} being similar).  First, we note that Theorem~\ref{mainellip} immediately yields an upper bound of 1.05 on the average rank of elliptic curves.  Indeed, recall that 
the $5$-Selmer group of an elliptic curve fits into the exact sequence
$$
0\to E(\Q) \to S_5(E) \to \SH_E[5] \to 0.
$$ 
If $r$ denotes the rank of an elliptic curve $E$, then the size of
the $5$-Selmer group of $E$ is an upper bound for $5^r$. Since
$20r-15\leq 5^r$ for any nonnegative integer $r$, we conclude by Theorem \ref{mainellip} 
that (the limsup of) the average rank $\bar r$ of elliptic curves, when ordered by height, must 
satisfy $20\bar r-15\leq 6$, whence $\bar r\leq 21/20$.  

To improve this bound further,
we observe that 
the bound of $1.05$ can be attained only if $95\%$ of elliptic curves
have rank $1$ and $5\%$ have rank $2$.  However, it is widely expected
that $50\%$ of elliptic curves should have even rank and $50\%$ should have odd
rank.  This is because the parity conjecture (implied by the Birch and
Swinnerton-Dyer conjecture) states that the rank of an elliptic curve
is even if and only if its root number is $+1$; furthermore, one expects that the root
number of elliptic curves should be
equidistributed. The parity conjecture has not been proven, but we may
instead use the remarkable result of Dokchitser and Dokchitser (\cite{DD}) which states that the parity of the $p$-Selmer
rank of an elliptic curve is determined by its root number.  This
result suffices for our purposes because Theorems \ref{mainellip} and
\ref{ellipcong} indeed yield bounds on not just the rank but the $5$-Selmer rank of elliptic curves.

Any result towards the equidistribution of root numbers of elliptic
curves would thus imply a better bound on the average rank.  However,
no cancellation in the root numbers of elliptic curves has been
established.  In~\cite{SW} and \cite{TC}, it was proved that there exist positive
proportion families of elliptic curves having equidistributed root number.

In this article, we prove that a {majority} of elliptic curves in fact
do have equidistributed root number:
\begin{theorem}\label{fam}
There exists a family $F$ of elliptic curves $E_{A,B}$, having density
greater than $55.01\%$ among all elliptic curves when ordered height\,
and defined by congruence conditions on $A$ and $B$, such that the
root number of elliptic curves in $F$ is equidistributed.
\end{theorem}
Specifically, we construct $F$ so that, for every elliptic
curve $E\in F$, the quadratic twist $E_{-1}$ of $E$ is also in $F$
and, moreover, $E$ and $E_{-1}$ have opposite root numbers.  We show
that we can find such an $F$ whose density in the family of all
elliptic curves over $\Q$ is $>55.01\%$.  The construction makes key
use, in particular, of the work of Rohrlich~\cite{Rh} and
Halberstadt~\cite{Hal} on computations of local root numbers.

To the special family $F$ constructed in Theorem~\ref{fam}, we may
then apply Theorem~\ref{ellipcong}, along with the aforementioned
theorem of Dokchitser--Dokchitser~\cite{DD}.  (This explains why
Theorem~\ref{ellipcong}---the congruence version of
Theorem~\ref{mainellip}---is also critical in the proof of
Theorem~\ref{thavgrank}.) Together, they imply that the average rank
of elliptic curves in $F$ is at most $.75$ (indeed, this upper bound
can be attained only if $37.5\%$ of curves in $F$ have $5$-Selmer size
1, $50\%$ have size 5, and $12.5\%$ have size~25).  This yields an
upper bound of
$$.5501 \times .75+ .4499 \times 1.05 < .885$$
on the average rank of elliptic curves, yielding
Theorem~\ref{thavgrank}.  Similar arguments are used to obtain
Theorems~\ref{rank01} and \ref{rank0}; see the last section \S6 for
details. 

This paper is organized as follows. In Section 2, we describe the parametrization of elements of 5-Selmer groups of elliptic curves using quintuples of $5\times 5$ skew-symmetric matrices.  This follows essentially from the work of Fisher, although we must slightly modify the group action so that our counting and sieve methods work more effectively.  In Section 3, we then count integral orbits of bounded height in this representation in terms of volumes of certain fundamental domains.  In Section 4, we carry out the necessary sieve to count only 5-Selmer elements of elliptic curves, thereby proving Theorem~\ref{mainellip}; we also obtain Theorem~\ref{ellipcong}.  

In Section 5, we then turn to root numbers, and construct the family $F$ above, thereby proving Theorem~\ref{fam}.  Finally, we complete the proofs of Theorems~\ref{thavgrank}, \ref{rank01}, and \ref{rank0} in Section~6.

\section{Parametrization of elements in the $5$-Selmer groups of elliptic curves}

For an elliptic curve $E:y^2=x^3+Ax+B$ over $\Q$, we define the
quantities $I(E)$ and $J(E)$ by
\begin{equation}\label{eqEIJ}
  \begin{array}{rcl}
    I(E)&=&-3A,\\[.02in]
    J(E)&=&-27B.
  \end{array}
\end{equation} 
There invariants are related
to the classical invariants $c_4$ and~$c_6$ in the following way: we
have $I=3^4c_4$ and $J=2\cdot 3^6c_6$. We denote the elliptic curve
with invariants $I$ and $J$ by $E^{I,J}$. 
The primary purpose of this section is to describe a method to parametrize
elements of the $5$-Selmer groups of elliptic curves over $\Q$. We
also describe similar parametrizations for elliptic curves over $\R$ and
$\Q_p$.

For any ring $R$, let $V_R$ denote the space $R^5\otimes\wedge^5(R)$ of
quintuples of skew-symmetric  $5\times 5$ matrices with coefficients in $R$.
The group $\GL_5(R)\times\GL_5(R)$ acts on $V_R$ via:
$$
(g_1,g_2)\cdot
(A,B,C,D,E):=(g_1Ag_1^t,g_1Bg_1^t,g_1Cg_1^t,g_1Dg_1^t,g_1Eg_1^t)\cdot g_2^t,
$$
for $(g_1,g_2)\in \GL_5(R)\times\GL_5(R)$ and $(A,B,C,D,E)\in V_R$.
We define the {\it determinant} of an element
$(g_1,g_2)\in\GL_5(R)\times\GL_5(R)$ by $\det(g_1,g_2):=(\det
g_1)^2\det g_2$. Let $G_R$ denote the group 
\begin{equation}\label{eqg}G_R:=\{(g_1,g_2)\in\GL_5(R)\times\GL_5(R):\det(g_1,g_2)=1\}/\{(\lambda I_5,\lambda^{-2} I_5)\},\end{equation}
where $I_5$ denotes the identity element of $\GL_5(R)$ and $\lambda\in
R^\times$. It is then easy to check that the action of
$\GL_5(R)\times\GL_5(R)$ on $V_R$ descends to an action of $G_R$ on
$V_R$.

The ring of invariants for the action of $G_\C$ on $V_\C$ is freely generated by
two elements (see, e.g., \cite{popovvinberg}). In \cite{Fisher1}, the generators
of the ring of invariants of the above action are denoted by~$c_4$ and~$c_6$, and they have degrees 20 and 30 on $V$, 
respectively. For purposes of convenience, we consider the invariants $I$ and
$J$ given by $3^4c_4$ and $2\cdot 3^6c_6$, respectively.  We
then define the {\it discriminant}~$\Delta(v)$  of an element $v\in V_R$ having invariants $I$ and $J$ by
\begin{equation*}
    \Delta(v) :=\Delta(I,J):=(4I^3-J^2)/27;
\end{equation*}
one checks that $\Delta(v)$ is an integer polynomial of degree 60 in the 50 entries of $V$.

Let $K$ be a field of characteristic not equal to $2$, $3$, or
$5$. Given $v=(A,B,C,D,E)\in V_K$ having nonzero discriminant, let $Q_1,\ldots,Q_5$ be the five $4\times 4$ sub-Pfaffians of $v\cdot(t_1,\ldots,t_5)=At_1+Bt_2+Ct_3+Dt_4+Et_5$, i.e., $Q_i$ is the Pfaffian of the $4\times 4$ matrix obtained by removing the $i$th row and column of 
The 
intersection of the quadrics $Q_i(t_1,t_2,t_3,t_4,t_5)=0$ in $\P^4$ is generically a genus one curve $\mathcal{C}_v$,
whose Jacobian is given by the elliptic curve
$E:y^2=x^3-\frac{I}{3}x-\frac{J}{27}$ (see \cite[Proposition 2.3]{Fisher1}).

An element $v\in V_K$ is called {\it $K$-soluble} if $\mathcal{C}_v$ has a $K$-rational point, i.e., $\mathcal{C}_v(K)\neq\emptyset$. If $C_v$ is $K$-soluble, then it corresponds naturally to an element of $E(K)/5E(K)$, where $E$ is the Jacobian of $C$.  This leads naturally to the following two results, which are essentially due to Fisher~(cf.\ \cite[Theorem~6.1]{Fisher3}); however, as in \cite[\S4.4]{BhHo}, we use a slightly different group action so that the stabilizer of the group action is given exactly by the $5$-torsion subgroup $E(K)[5]$ of $E(K)$; this will be important in our applications.

\begin{theorem}\label{leme2ep}{\bf (\cite[Thm.\ 6.1]{Fisher3}, \cite[\S4.4]{BhHo})}
  Let $K$ be a field of characteristic not equal to $2$, $3$, or $5$. Let
  $E=E^{I,J}:y^2=x^3-\frac{I}{3}x-\frac{J}{27}$ be an elliptic curve over $K$.
  Then there exists a canonical bijection between elements of $E(K)/5E(K)$ and
  $G_K$-orbits of $K$-soluble elements in $V_K$ having invariants $I$
  and~$J$.
  Furthermore, the stabilizer in $G_K$ of any $($not necessarily
  $K$-soluble$)$ element in $V_K$, having invariants $I$ and $J$, is
  isomorphic to $E(K)[5]$.
\end{theorem}

An element in $V_\Q$ is called {\it locally soluble} if it is
$\R$-soluble and $\Q_p$-soluble for all primes $p$.  

\begin{theorem}[\cite{Fisher3}, \cite{BhHo}]
\label{propselparam}
  Let $E/\Q$ be an elliptic curve. Then the elements in the $5$-Selmer
  group~of~$E$ are in bijective correspondence with $G_\Q$-equivalence
  classes on the set of locally soluble elements in~$V_\Q$ having
  invariants equal to $I(E)$ and $J(E)$.
\end{theorem}

If $v\in V_{\Q_p}$ is $\Q_p$-soluble and has integral invariants, then it is
a further result of Fisher~\cite{Fishermin} that $v$ is
$G_{\Q_p}$-equivalent to an element in $V_{\Z_p}$.

\begin{theorem}\label{proplocalmin}{\bf (\cite[Thm.~2.1]{Fishermin})}
Let $p$ be a prime and let $v\in V_{\Q_p}$ be an element having integral invariants that is soluble over $\Q_p$. Then $v$ is $G_{\Q_p}$-equivalent to an element in $V_{\Z_p}$.
\end{theorem}
Since $G_\Q$ has class number 1, Theorems \ref{propselparam} and \ref{proplocalmin} immediately imply:
\begin{theorem}[\cite{Fishermin}]\label{propselparz}
  Let $E/\Q$ be an elliptic curve. Then the elements in the $5$-Selmer
  group~of~$E$ are in bijective correspondence with $G_\Q$-equivalence
  classes on the set of locally soluble elements in~$V_\Z$ having
  invariants equal to $I(E)$ and $J(E)$.
\end{theorem}

Finally, we will need the following proposition, also due to Fisher~\cite{Fisher2}.  It states that any element $v\in V_{\Z_p}$ for which $p^2$ does not divide the discriminant is automatically $\Q_p$-soluble; moreover, the stabilizer of $v$ in $G_{\Z_p}$ is the same as that in $G_{\Q_p}$, and the notion of $G_{\Z_p}$-equivalence of such elements $v$ is the same as that of $G_{\Q_p}$-equivalence:

\begin{proposition}\label{Fisherprop}{\bf (\cite{Fisher2})}
Let $p$ be any prime and let $v\in V_{\Z_p}$ be an element such that $p^2\nmid \Delta(v)$. Then
$$v\mbox{ is $\Q_p$-soluble};\quad\Stab_{G_{\Q_p}}(v)=\Stab_{G_{\Z_p}}(v);\quad (G_{\Q_p}\cdot v)\cap V_{\Z_p}=G_{\Z_p}\cdot v.$$
\end{proposition}

\section{Counting orbits of bounded height}

We write elements in $V_\R=\R^5\otimes\wedge^2\R^5$ as quintuples
$(A,B,C,D,E)$ of skew-symmetric $5\times 5$ matrices, where the
matrices $A,\;B,\;C,\;D$, and $E$ have entries $a_{ij}$, $b_{ij}$,
$c_{ij}$, $d_{ij}$, and $e_{ij}$, respectively, with $1\leq i<j\leq5$.
We define the {\it height} of an element $v\in V_\R$ having invariants
$I$ and $J$ to be
\begin{equation}\label{heightvz}
H(v):=H(I,J):=\max\{|I|^3,J^2/4\}
\end{equation}
which is clearly $G_\R$-invariant.

The action of $G_\Z$ preserves the lattice $V_\Z\subset V_\R$
consisting of the quintuples of skew-symmetric $5\times 5$ matrices
whose entries are integral. In fact, it also preserves the two sets
$\pv$ and $\nv$ consisting of elements in $V_\Z$ having positive and
negative discriminant, respectively.  We say that an element $v\in
V_\Z$ having invariants $I$ and $J$ is {\it strongly irreducible} if
it has nonzero discriminant and does not correspond to the identity
element in the $5$-Selmer group of $E^{I,J}$, i.e., the $5$-covering
$\mathcal{C}_v$ corresponding to $v$ is the not the trivial
$5$-covering of $E^{I,J}$.

In this section, we compute asymptotics for the number of
$G_\Z$-orbits on strongly irreducible elements in $\pnv$ having
bounded height. To state the result, 
let $N^+(X)$ (resp.\ $N^-(X)$) denote the number of pairs
$(I,J)\in\Z\times\Z$ having height less than $X$ and positive
(resp.\ negative) discriminant.  For any $G_\Z$-invariant set
$S\subset V_\Z$, let $N(S;X)$ denote the number of $G_\Z$-orbits on
strongly irreducible elements in $S$ having height less than
$X$. Finally, throughout this paper, we fix $\omega$ to be a
differential that generates the rank 1 module of top-degree
left-invariant differential forms of $G$ over $\Z$.

Then we prove the following theorem:
\begin{theorem}\label{thsec3main}
There exists a nonzero rational constant $\J$, to be defined in Proposition $\ref{propjac}$, such that
$$N(\pnv;X)=|\J|\cdot\Vol(G_\Z\backslash G_\R)\cdot N^\pm(X)+o(X^{5/6}),$$
where the volume of $G_\Z\backslash G_\R$ is computed with respect to
$\omega$.
\end{theorem}
\begin{rem}
{\rm  It follows from {\rm \cite[Equations (24),(25)]{BS}} that, up to an
  error of $O(X^{1/2})$, we have $N^+(X)=\frac85X^{5/6}$ and
  $N^-(X)=\frac{32}5X^{5/6}$. Thus the error term in the equation of
  the above theorem is indeed smaller than the main term.}
\end{rem}

\subsection{Reduction theory}\label{s31}
Let $\pr$ and $\nr$ denote the set of elements in $V_\R$ having
positive and negative discriminant, respectively. The purpose of this
subsection is to construct finite covers of fundamental domains for
the action of $G_\Z$ on $\pnr$. We first construct fundamental
sets for the action of $G_\R$ on $\pnr$.

To apply Theorem~\ref{leme2ep} in the case of $G_\R$-orbits on
$\pnr$, we need the following lemma:
\begin{lemma}\label{lemalwaysrsol}
  Let $v\in V_\R$ have nonzero discriminant. Then $v$ is $\R$-soluble.
\end{lemma}
\begin{proof}
  Let $\mathcal{C}_v$ be the curve corresponding to $v$ under the
  correspondence described in Section~2. It is known (see \cite[Remark 1.23]{algebra})
  that $\mathcal{C}_v$ has $25$ flex points defined over $\C$. Since
  $v\in V_\R$, these flex points come~in complex conjugate
  pairs. Thus, at least one of them is defined over $\R$, implying
  that $\mathcal{C}_v(\R)\neq \emptyset$.
\end{proof}

The next
proposition follows from Theorem~\ref{leme2ep}, Lemma \ref{lemalwaysrsol}, and the fact that the group 
$E(\R)/5E(\R)$ is trivial for every elliptic curve $E$ over $\R$:
\begin{proposition}\label{propgrvr}
  Let $(I,J)\in\R\times\R$ be such that $\Delta(I,J)\neq 0$. Then the
  set of elements in $V_\R$ having invariants $I$ and $J$
  consists of one soluble $G_\R$-orbit.
\end{proposition}
Thus we may construct fundamental sets
$\pnrg$ for the action of $G_\R$ on $V_\R^\pm$
by choosing one element $v\in V_\R^\pm$
having invariants $I$ and $J$ for each pair
$(I,J)\in\R\times\R$ such that $\Delta(I,J)\in\R^\pm$.
We now choose specific sets $\pnrg$.
The element $v_{I,J}$ defined by
\begin{equation}\label{eqxij}
  v_{I,J}:=\left[
\left(\begin{smallmatrix}{0}&{\frac{-J}{27}}&{0}&0&{0}\\ {\frac{J}{27}}&{0}&{0}&{0}&{0}\\{0}&{0}&{0}&{0}&{0}\\{0}&{0}&{0}&{0}&{1}\\{0}&{0}&{0}&{-1}&{0}\end{smallmatrix}\right),
\left(\begin{smallmatrix}{0}&{{\frac{-I}{3}}}&{0}&{0}&{0}\\ {\frac{I}{3}}&{0}&{0}&{0}&{1}\\{0}&{0}&{0}&{-1}&{0}\\{0}&{0}&{1}&{0}&{0}\\{0}&{-1}&{0}&{0}&{0}\end{smallmatrix}\right),
\left(\begin{smallmatrix}{0}&{0}&{0}&{0}&{1}\\ {0}&{0}&{0}&{1}&{0}\\{0}&{0}&{0}&{0}&{0}\\{0}&{-1}&{0}&{0}&{0}\\{-1}&{0}&{0}&{0}&{0}\end{smallmatrix}\right),
\left(\begin{smallmatrix}{0}&{0}&{0}&{1}&{0}\\ {0}&{0}&{1}&{0}&{0}\\{0}&{-1}&{0}&{0}&{0}\\{-1}&{0}&{0}&{0}&{0}\\{0}&{0}&{0}&{0}&{0}\end{smallmatrix}\right),
\left(\begin{smallmatrix}{0}&{0}&{1}&{0}&{0}\\ {0}&{0}&{0}&{0}&{0}\\{-1}&{0}&{0}&{0}&{0}\\{0}&{0}&{0}&{0}&{0}\\{0}&{0}&{0}&{0}&{0}\end{smallmatrix}\right)
\right],
\end{equation}
has invariants equal to $I$ and $J$ by \cite[\S 6]{Fisher1}. We define $\pnrg$ by
\begin{equation}\label{eqrij}
  \begin{array}{rcl}
   \prg&:=&\{\lambda\cdot v_{I,J}:\lambda\in\R_{>0},\;H(I,J)=1,\;\Delta(I,J)>0\};\\[.1in]
   \nrg&:=&\{\lambda\cdot v_{I,J}:\lambda\in\R_{>0},\;H(I,J)=1,\;\Delta(I,J)<0\}.
  \end{array}
\end{equation}

Since $H(\lambda\cdot v)=\lambda^{60}H(v)$ (as $I(v)$ and $J(v)$ are
polynomials in the coefficients of $v$ having degrees $20$ and $30$,
respectively), we see that the coefficients of all the elements in
$\pnrg$ having height less than $X$ are bounded by $O(X^{1/60})$. Note
also that for any $g\in G_\R$, the set $g\cdot\pnrg$ is also a
fundamental set for the action of $G_\R$ on $\pnr$. Furthermore, for
any compact set $G_0\subset G_\R$, the coefficients of elements in
$g\cdot\pnrg$, with $g\in G_0$, having height less than $X$, is
bounded by $O(X^{1/60})$, where the implied constant depends only on
$G_0$.

Let $\FF$ denote a fundamental domain for the left action of $G_\Z$ on
$G_\R$ that is contained in a standard Siegel set~\cite[\S2]{BH}. We may assume that
$\FF=\{nak:n\in N'(a),a\in A',k\in K\}$, where
\begin{eqnarray*}
  &\!\!\!\!\!\!\!\!\!K\!\!\!\!\!\!\!&\!=\{{\rm subgroup\; of \;orthogonal\; transformations\; } 
\SO_5(\R)\times\SO_5(\R)\subset G_\R\};\\
  &\!\!\!\!\!\!\!\!\!A'\!\!\!\!\!\!\!&\!=\{a(s_1,s_2,s_3,s_4,s_5,s_6,s_7,s_8):s_i,t_i>c\},\\
  &&\;\;\;\;\;\;\;\;{\rm \!\!\!\!\!\!\!\!\!where}\;a(s_1,s_2,s_3,s_4,s_5,s_6,s_7,s_8)\!=\!
\left[\left(\begin{smallmatrix}
    {s_1^{-4}s_2^{-3}s_3^{-2}s_4^{-1}} & {}&{}&{}&{}\\ {} & \!\!\!\!\!\!\!\!\!\!\!\!\!\!\!\!\!\!\!\!\!\!\!\!\!\!{s_1s_2^{-3}s_3^{-2}s_4^{-1} }&{}&{}&{}\\
{} & {}&\!\!\!\!\!\!\!\!\!\!\!\!\!\!\!\!\!\!\!\!\!\!\!\!\!\!{s_1s_2^{2}s_3^{-2}s_4^{-1}}&{}&{}\\
{} & {}&{}&\!\!\!\!\!\!\!\!\!\!\!\!\!\!\!\!\!\!\!\!\!{s_1s_2^{2}s_3^{3}s_4^{-1}}&{}\\
{}&{} & {}&{}&\!\!\!\!\!\!\!\!\!\!\!\!\!\!\!\!\!{s_1s_2^{2}s_3^{3}s_4^{4}}
\end{smallmatrix}\right),
\left(\begin{smallmatrix}
    {s_5^{-4}s_6^{-3}s_7^{-2}s_8^{-1}} & {}&{}&{}&{}\\ {} & \!\!\!\!\!\!\!\!\!\!\!\!\!\!\!\!\!\!\!\!\!\!\!\!\!\!{s_5s_6^{-3}s_7^{-2}s_8^{-1} }&{}&{}&{}\\
{} & {}&\!\!\!\!\!\!\!\!\!\!\!\!\!\!\!\!\!\!\!\!\!\!\!\!\!\!{s_5s_6^{2}s_7^{-2}s_8^{-1}}&{}&{}\\
{} & {}&{}&\!\!\!\!\!\!\!\!\!\!\!\!\!\!\!\!\!\!\!\!\!{s_5s_6^{2}s_7^{3}s_8^{-1}}&{}\\
{}&{} & {}&{}&\!\!\!\!\!\!\!\!\!\!\!\!\!\!\!\!\!{s_5s_6^{2}s_7^{3}s_8^{4}}
\end{smallmatrix}\right)\right];\\
&\!\!\!\!\!\!\!\!\!N'\!\!\!\!\!\!\!&\!=\{n(u_1,\cdots,u_{20}):(u_i)\in\nu(a)\},\;\\
&&\;\;\;\;\;\;\;\;{\rm \!\!\!\!\!\!\!\!\!where}\; n(u)\!=\!\left[\left(\begin{smallmatrix}
{1} & {}&{}&{}&{}\\ {u_1} & {1}&{}&{}&{}\\{u_2}&{u_3}&{1}&{}&{}\\{u_4} & {u_5}&{u_6}&{1}&{}\\{u_7} & {u_8}&{u_{9}}&{u_{10}}&{1}
\end{smallmatrix}\right),
\left(\begin{smallmatrix}
{1} & {}&{}&{}&{}\\ {u_{11}} & {1}&{}&{}&{}\\{u_{12}}&{u_{13}}&{1}&{}&{}\\{u_{14}} & {u_{15}}&{u_{16}}&{1}&{}\\{u_{17}} & {u_{18}}&{u_{19}}&{u_{20}}&{1}
\end{smallmatrix}\right)\right];
\end{eqnarray*}
here $\nu(a)$ is a measurable subset of $[-1/2,1/2]^{20}$ dependent
only on $a\in A'$, and $c>0$ is an absolute constant. We now require
the following result that follows from Theorem~\ref{leme2ep} and the
fact that every elliptic curve over $\R$ has exactly five $5$-torsion
points defined over $\R$.
\begin{lemma}\label{lemstabsize}
  Let $v\in V_\R$ be any element having nonzero discriminant. Then the
  size of the stabilizer in $G_\R$ of $v$ is equal to $5$.
\end{lemma}

For $h\in G_\R$, we regard $\FF h\cdot R^\pm$
as a multiset, where the multiplicity of an element $v\in V_\R$ is
equal to $\#\{g\in\FF:v\in gh\cdot R^\pm\}$.  By an argument identical to
that in \cite[\S2.1]{BS}, it follows that for any $h\in G_\R$ and any
$v\in V_\R^\pm$, the $G_\Z$-orbit of $v$ is represented $m(v)$ times in $\FF h\cdot R^\pm$, where
$$
m(v):=\#\Stab_{G_\R}(v)/\#\Stab_{G_\Z}(v).
$$ That is, the multiplicity of $v'$ in $\FF h\cdot R^\pm$,
summed over all $v'$ that are $G_\Z$-equivalent to $v$, is equal to
$m(v)$.

The set of elements in $V_\R^\pm$ that have a nontrivial stabilizer
in $G_\Z$ has measure $0$ in $V_\R^\pm$. Thus, by
Lemma~\ref{lemstabsize}, for any $h\in G_\R$ the multiset $\FF
h\cdot R^\pm$ is a $5$-fold cover of a fundamental domain for the
action of $G_\Z$ on $V_\R^\pm$.

Let $R^\pm(X)$ denote the set of
elements in $R^\pm$ having height less than 
$X$. It then follows, for any $G_\Z$-invariant
subset $S\subset V_\Z$, that
$5N(S;X)$ is equal to the number of strongly irreducible elements
in $\FF h\cdot R^\pm(X)\cap S$, with the slight caveat that the
(relatively rare---see Lemma \ref{lemtotred}) elements
with $G_\Z$ stabilizers of size $r$ ($r>1$) are counted with weight
$1/r$.

Counting strongly irreducible integer points in a single such domain $\FF
h\cdot\pnrg(X)$ is difficult because this domain is unbounded (although we 
will show it has finite volume). As in \cite{BS}, we simplify the counting by averaging over
lots of such domains, i.e., by averaging over a continuous range of
elements $g$ lying in a compact subset of $G_\R$.

\subsection{Averaging}

Let $G_0\subset G_\R$ be a compact, semialgebraic, left $K$-invariant set that is the
closure of an open nonempty set. Let $dh$ be the Haar measure on
$G_\R$ normalized as follows: we set $dh=dn\,d^\ast a\,dk$,
where $n$, $a$, $dn$, and $d^\ast a$ are given by $n=n(u_1,\ldots,u_{20})$,
$a=a(s_1,s_2,s_3,s_4,s_5,s_6,s_7,s_8)$, $dn=du_1\cdots du_{20}$, and 
\begin{equation}\label{dadn}
d^\ast a\,=\,s_1^{-20}s_2^{-30}s_3^{-30}s_4^{-20}s_5^{-20}s_6^{-30}s_7^{-30}s_8^{-20}d^\times
s_1\,d^\times s_2\,d^\times s_3\,d^\times s_4\,d^\times s_5\,d^\times
s_6\,d^\times s_7\,d^\times s_8,
\end{equation}
respectively, and $dk$ is Haar measure on $K$ normalized so that $K$
has volume $1$.

For any $G_\Z$-invariant set $S\subset V_\Z^\pm$, the arguments of
\S\ref{s31} imply that we have
\begin{equation}\label{eqavgfirst}
  N(S;X)=\displaystyle\frac{\int_{h\in G_0}\#\{\FF h\cdot\pnrg(X)\cap S^\irr\}dh}{C_{G_0}},
\end{equation}
where $S^\irr$ denotes the set of strongly irreducible elements in
$S$, and $C_{G_0}:=5\int_{h\in G_0}dh$.  We take the right hand side of \eqref{eqavgfirst} as the
definition of $N(S;X)$ also for sets $S$ that are not necessarily
$G_\Z$-invariant.  By an argument identical to the proof of
\cite[Theorem 2.5]{BS}, we see that the right hand side of
\eqref{eqavgfirst} is equal to
\begin{equation}\label{eqavgfinal}
  N(S;X)=\frac{1}{C_{G_0}}\int_{na\in\FF}\#\{B^\pm(n,a;X)\cap S^\irr\}dnd^\ast a,
\end{equation}
where $B^\pm(n,a;X)$ denotes the multiset $naG_0\cdot\pnrg(X)$. 

To estimate the number of integral points in $B^\pm(n,a;X)$, we use the following proposition due to Davenport.
\begin{proposition}\label{davlem}
  Let $\mathcal R$ be a bounded, semi-algebraic multiset in $\R^n$
  having maximum multiplicity $m$, and that is defined by at most $k$
  polynomial inequalities each having degree at most $\ell$.  
  Then the number of integer lattice points $($counted with
  multiplicity$)$ contained in the region $\mathcal R$ is
\[\Vol(\mathcal R)+ O(\max\{\Vol(\bar{\mathcal R}),1\}),\]
where $\Vol(\bar{\mathcal R})$ denotes the greatest $d$-dimensional 
volume of any projection of $\mathcal R$ onto a coordinate subspace
obtained by equating $n-d$ coordinates to zero, where 
$d$ takes all values from
$1$ to $n-1$.  The implied constant in the second summand depends
only on $n$, $m$, $k$, and $\ell$.
\end{proposition}

Proposition \ref{davlem} yields good estimates on the number of
integral points in $B^\pm(n,a;X)$, for $a=a(s_1,\ldots,s_8)$, when the
$s_i$'s are ``small'' compared to $X$. However, when any of the
$s_i$'s are ``large'', the error term in Proposition \ref{davlem}
dominates the main term. To resolve this issue, in \S\ref{s33} we divide the
fundamental domain into a ``main body'' and a ``cuspidal region''. Proposition \ref{davlem} will yield good estimates on the
number of integral points in the main body. We then bound the number
of strongly irreducible elements in the cuspidal region, the volume of
the cuspidal region, and the number of reducible elements in the main
body.  These results together then allow us to deduce that
$N(V^\pm_\Z;X)$ is well-approximated by the volume $\FF\cdot
R^\pm(X)$. Finally, in~\S3.4, we compute the volume of the region
$\FF\cdot R^\pm(X)$, thus completing the proof of Theorem
\ref{thsec3main}.

\subsection{Conditions on reducibility and cutting off the cusp}\label{s33}

Our first aim in the subsection is to prove that the number of
strongly irreducible points in the ``cuspidal region'' of the
fundamental domain is negligible:
\begin{proposition}\label{propaneq0}
  Let $V^\irr_\Z(0)$ denote the set of strongly irreducible points $(A,B,C,D,E)\in V_\Z$ satisfying $a_{12}=0$. Then
$N(V^\irr_\Z(0);X)=O(X^{499/600})$.
\end{proposition}

We begin by describing sufficient conditions to ensure that
points $(A,B,C,D,E)\in V_\Z$ are not strongly irreducible.

\begin{lemma}\label{lemcondredprelim}
  Let $v=(A,B,C,D,E)$ be an element in $V_\Z$ and let
  $Q_1,\;Q_2,\;Q_3,\;Q_4$, and $Q_5$ denote the five $4\times 4$ sub-Pfaffians of $At_1+Bt_2+Ct_3+Dt_4+Et_5$. Then
  \begin{itemize}
  \item[{\rm (a)}] Let $\mathcal{C}_v$ be the curve in $\P^4$ defined by
    $Q_i(t_1,t_2,t_3,t_4,t_5)=0$ for all $i\in\{1,\ldots,5\}$. If $\mathcal{C}_v$ is not a smooth genus one curve, then the
    discriminant of $v$ is $0$.
  \item[{\rm (b)}] If $Q_1$ is reducible over $\overline{\Q}$ $($i.e., $Q_1$ 
  factors into a product of linear forms over $\overline{\Q})$, then the discriminant of $v$ is $0$.
  \item[{\rm (c)}] Let $Q_1',\;Q_2',\;Q_3',\;Q_4'$, and $Q_5'$ be the
    quadratic forms in four variables obtained from $Q_1$, $Q_2$,
    $Q_3$, $Q_4$, and $Q_5$, respectively, by setting $t_5=0$. If the
    intersection of the quadrics $Q_i'(t_1,t_2,t_3,t_4)=0$ in $\P^3(\overline{\Q})$ consists of a single point, then $v$ corresponds to the identity element in the $5$-Selmer group of $E^{I(v),J(v)}$. 
  \end{itemize}
\end{lemma}
\begin{proof}
Parts (a) and (b) follow from \cite[Theorem~4.4(ii)]{Fisher1} and \cite[Theorem~5.10(i)]{Fisher1}, respectively.
For Part (c), note that an element in the $5$-Selmer group of an elliptic curve $E$ may viewed as a torsor for $E[5]$ (see, e.g., \cite[\S 1.4]{algebra}). By 
\cite[Remarks 1.15, 1.20, and 1.23]{algebra}, it follows that given
$v\in V_\Q$, the corresponding torsor for $E[5]$ is obtained by taking the set
of points $P$ on $C_v$ such that $5\cdot P$ is linearly equivalent to
$D$, where $D$ is the hyperplane divisor corresponding to a fixed
rational hyperplane section $H$. If the hypotheses of Part (c) are satisfied, then $C_v$ intersects the hyperplane section
$H$ given by $t_5=0$ in a rational point $P$ with multiplicity
$5$. Therefore the torsor for $E[5]$ contains a rational point, namely
$P$. It follows that $C_v$ is the trivial $5$-covering of $E$. 
\end{proof}

\begin{proposition}\label{propreducibility}
  Let $v=(A,B,C,D,E)\in V_\Z$ be such that all the variables in at least one of the following sets vanish:
  \begin{itemize}
  \item[{\rm (1)}]$\{a_{12},a_{13}\}\cup\{b_{12},b_{13}\}\cup\{c_{12},c_{13}\}\cup\{d_{12},d_{13}\}\cup\{e_{12},e_{13}\}$
  \item[{\rm (2)}]$\{a_{12},a_{13},a_{14}\}\cup\{b_{12},b_{13},b_{14}\}\cup\{c_{12},c_{13},c_{14}\}\cup\{d_{12},d_{13},d_{14}\}$
  \item[{\rm (3)}]$\{a_{12},a_{13},a_{23}\}\cup\{b_{12},b_{13},b_{23}\}\cup\{c_{12},c_{13},c_{23}\}\cup\{d_{12},d_{13},d_{23}\}$
  \item[{\rm (4)}]$\{a_{12},a_{13},a_{14},a_{15},a_{23},a_{24},a_{25}\}\cup\{b_{12},b_{13},b_{14},b_{15},b_{23},b_{24},b_{25}\}$
  \item[{\rm (5)}]$\{a_{12},a_{13},a_{14},a_{23},a_{24},a_{34}\}\cup\{b_{12},b_{13},b_{14},b_{23},b_{24},b_{34}\}$
   \item[{\rm (6)}]$\{a_{12},a_{13},a_{14},a_{23},a_{24}\}\cup\{b_{12},b_{13},b_{14},b_{23},b_{24}\}\cup\{c_{12},c_{13},c_{14},c_{23},c_{24}\}$
 \item[{\rm (7)}]$\{a_{12},a_{13},a_{14},a_{15},a_{23},a_{24},a_{25},a_{34},a_{35},a_{45}\}$
 \item[{\rm (8)}]$\{a_{12},a_{13},a_{14},a_{15}\}\cup\{b_{12},b_{13},b_{14},b_{15}\}\cup\{c_{12},c_{13},c_{14},c_{15}\}$

  \item[{\rm (9)}]$\{a_{12},a_{13},a_{14},a_{15},a_{23},a_{24},a_{25}\}\cup\{b_{12}\}\cup\{c_{12}\}\cup\{d_{12}\}\cup\{e_{12}\}$
  \item[{\rm (10)}]$\{a_{12},a_{13},a_{14},a_{15},a_{23},a_{24}\}\cup\{b_{12},b_{13},b_{14},b_{15},b_{23},b_{24}\}\cup\{c_{12}\}\cup\{d_{12}\}\cup\{e_{12}\}$

  \item[{\rm (11)}]$\{a_{12},a_{13},a_{14},a_{15},a_{23},a_{24},a_{25},a_{34},a_{35}\}\cup\{b_{12},b_{13},b_{23}\}\cup\{c_{12},c_{13},c_{23}\}$
  \item[{\rm (12)}]$\{a_{12},a_{13},a_{14},a_{15},a_{23},a_{24},a_{25},a_{34},a_{35}\}\cup\{b_{12},b_{13}\}\cup\{c_{12},c_{13}\}\cup\{d_{12},d_{13}\}$

  \item[{\rm (13)}]$\{a_{12},a_{13},a_{14},a_{15},a_{23},a_{24},a_{25},a_{34}\}\cup\{b_{12},b_{13},b_{14},b_{15},b_{23},b_{24}\}\cup\{c_{12},c_{13},c_{14},c_{23}\}\cup\{d_{12},d_{13}\}$
  \end{itemize}
  Then $(A,B,C,D,E)$ is not strongly irreducible.
\end{proposition}
\begin{proof}
  We first prove that if any of the first twelve cases holds, then the
  discriminant of $(A,B,C,D,E)$ is zero.
  Let $Q_1,\;Q_2,\;Q_3,\;Q_4$, and $Q_5$ denote the five $4\times
  4$ sub-Pfaffians of $At_1+Bt_2+Ct_3+Dt_4+Et_5$. 
In Case (1), the quadratic form corresponding to $Q_1$ is
$$-(a_{14} t_1 + b_{14} t_2 + c_{14} t_3+d_{14} t_4 + e_{14} t_5) (a_{23}x_1 +b_{23}t_2+c_{23}t_3+d_{23} t_4 + e_{23} t_5).$$
In Cases (2) and (3), the Gram matrix of $Q_1$ is of the form
$$Q_1=\left(\begin{smallmatrix}
        {0}&{0}&{0}&{0}&*\\
        {0}&{0}&{0}&{0}&*\\
        {0}&{0}&{0}&{0}&*\\
        {0}&{0}&{0}&{0}&*\\
        {*}&{*}&{*}&{*}&*\end{smallmatrix}\right).$$
Thus, in the first three cases the quadratic form corresponding to
  $Q_1$ factors into two linear factors, and hence by Lemma~\ref{lemcondredprelim}(b) 
 we have that $\Delta(A,B,C,D,E)=0$. 

  In Cases (4) and (5), the $t_1^2$-, $t_1t_2$-, and $t_2^2$- coefficients of
  all the $Q_i$'s are equal to zero, and so the Gram matrices of the $Q_i'$'s take the form
$$
[Q_1,Q_2,Q_3,Q_4,Q_5]=\left[\left(\begin{smallmatrix}
        {0}&{0}&{*}&{*}&*\\
        {0}&{0}&{*}&{*}&*\\
        {*}&{*}&{*}&{*}&*\\
        {*}&{*}&{*}&{*}&*\\
        {*}&{*}&{*}&{*}&*\end{smallmatrix}\right),\left(\begin{smallmatrix}
        {0}&{0}&{*}&{*}&*\\
        {0}&{0}&{*}&{*}&*\\
        {*}&{*}&{*}&{*}&*\\
        {*}&{*}&{*}&{*}&*\\
        {*}&{*}&{*}&{*}&*\end{smallmatrix}\right),\left(\begin{smallmatrix}
        {0}&{0}&{*}&{*}&*\\
        {0}&{0}&{*}&{*}&*\\
        {*}&{*}&{*}&{*}&*\\
        {*}&{*}&{*}&{*}&*\\
        {*}&{*}&{*}&{*}&*\end{smallmatrix}\right),\left(\begin{smallmatrix}
        {0}&{0}&{*}&{*}&*\\
        {0}&{0}&{*}&{*}&*\\
        {*}&{*}&{*}&{*}&*\\
        {*}&{*}&{*}&{*}&*\\
        {*}&{*}&{*}&{*}&*\end{smallmatrix}\right),\left(\begin{smallmatrix}
        {0}&{0}&{*}&{*}&*\\
        {0}&{0}&{*}&{*}&*\\
        {*}&{*}&{*}&{*}&*\\
        {*}&{*}&{*}&{*}&*\\
        {*}&{*}&{*}&{*}&*\end{smallmatrix}\right)\right].
$$
This implies that the curve $C_v$ defined by the vanishing of the $Q_i$ is not
  a smooth genus one curve, as it contains a projective line in $\P^4$.  Meanwhile, in Case (6) we may further
  assume (replacing $(A,B,C,D,E)$ with a $G_\C$-translate, if
  necessary) that $a_{34}=b_{34}=0$. Now the $t_1^2$-, $t_1t_2$-, and $t_2^2$-
  coefficients of all the $Q_i$'s are again equal to zero, as before.  Thus,
  Lemma~\ref{lemcondredprelim}(a) implies that the discriminant of
  $(A,B,C,D,E)$ is $0$ in these three cases.

  To prove Cases (7) through (12), 
  note that
  the discriminant of $(A,B,C,D,E)$ is a degree~$60$ polynomial in the
  $a_{ij}$'s, $b_{ij}$'s, $c_{ij}$'s, $d_{ij}$'s, and $e_{ij}$'s. Let
  $m$ be any monomial summand of the discriminant polynomial. We
  define $a(m)$ to be the number of factors $a_{ij}$
  ($i,j\in\{1,2,3,4,5\}$, $i>j$) that occur in $m$, and $1(m)$ to be
  the number of factors $x_{1j}$ ($x\in\{a,b,c,d,e\}$,
  $j\in\{2,3,4,5\}$) that occur in $m$, counted with multiplicity. We
  similarly define $b(m)$, $c(m)$, $d(m)$, $e(m)$, $2(m)$, $3(m)$,
  $4(m)$, and $5(m)$.  Classical invariant theory implies that
  \begin{equation}\label{eqdisccond}
    \begin{array}{ccccccccccc}
      a(m)\!\!\!&=&\!\!\!b(m)\!\!\!&=&\!\!\!c(m)\!\!\!&=&\!\!\!d(m)\!\!\!&=&\!\!\!e(m)\!\!\!&=&\!\!\!12,\\
      1(m)\!\!\!&=&\!\!\!2(m)\!\!\!&=&\!\!\!3(m)\!\!\!&=&\!\!\!4(m)\!\!\!&=&\!\!\!5(m)\!\!\!&=&\!\!\!24.
    \end{array}
  \end{equation}

  From these observations we see that an element in $V_\Z$ having all
  the $a_{ij}$ equal to zero has discriminant zero, since every
  monomial term in the discriminant polynomial has some $a_{ij}$ as a
  factor. Thus, Case (7) follows.

  Now suppose that there exists some element satisfying the condition
  of Case (8) and having nonzero discriminant. Then the discriminant
  polynomial must have a monomial summand~$m$ with no factor of the
  form $a_{1j}$, $b_{1j}$, or $c_{1j}$.  Then (\ref{eqdisccond})
  implies that every factor $d_{ij}$ or $e_{ij}$ of this summand~$m$
  satisfies $i=1$.  We claim that Case~(5) shows that this is
  impossible. Indeed, the truth of Case~(5) implies that the discriminant
  polynomial cannot have a nonzero monomial summand in which every factor $a_{ij}$ or
  $b_{ij}$ has $j=5$. The same argument then applies with $a$, $b$, and
  $5$ replaced by $d$, $e$, and~$1$, respectively. Thus, Case (8)
  follows.

  In Case (9), a nonzero monomial summand $m$ of the discriminant
  polynomial cannot have a factor of the form $b_{34}$, $b_{35}$, $b_{45}$,
  $c_{34}$, $c_{35}$, $c_{45}$, $d_{34}$, $d_{35}$, $d_{45}$,
  $e_{34}$, $e_{35}$, or $e_{45}$. (Otherwise $1(m)+2(m)$ would be
  strictly smaller than $48$, contradicting \eqref{eqdisccond}.)  Case~(9) now follows from Case (3) just as Case~(8) followed from Case~(5). 
  Case~(10) follows immediately from Case~(9) since any element
  satisfying the conditions of Case~(10) is $G_\Z$-equivalent to one
  satisfying the conditions of Case~(9).

  We turn to Case (11). A nonzero monomial summand $m$ of the
  discriminant polynomial cannot have a factor of the form $d_{i4}$,
  $e_{i4}$, $d_{i5}$, or $e_{i5}$ (otherwise, $4(m)+5(m)$ would be
  greater than $48$). Case~(4) shows that no such summand exists.
  Case~(12) follows from Case~(11) just as Case~(10) followed from
  Case~(9).

  Finally, in Case (13), the Gram matrices of the $Q_i$'s have the following form:
$$[Q_1,Q_2,Q_3,Q_4,Q_5]=\left[\left(\begin{smallmatrix}
        {0}&{0}&{0}&{0}&0\\
        {0}&{0}&{0}&{0}&*\\
        {0}&{0}&{0}&{0}&*\\
        {0}&{0}&{0}&{*}&*\\
        {0}&{*}&{*}&{*}&*\end{smallmatrix}\right),\left(\begin{smallmatrix}
        {0}&{0}&{0}&{0}&*\\
        {0}&{0}&{0}&{0}&*\\
        {0}&{0}&{0}&{*}&*\\
        {0}&{0}&{*}&{*}&*\\
        {*}&{*}&{*}&{*}&*\end{smallmatrix}\right),\left(\begin{smallmatrix}
        {0}&{0}&{0}&{0}&*\\
        {0}&{0}&{0}&{*}&*\\
        {0}&{0}&{*}&{*}&*\\
        {0}&{*}&{*}&{*}&*\\
        {*}&{*}&{*}&{*}&*\end{smallmatrix}\right),\left(\begin{smallmatrix}
        {0}&{0}&{0}&{*}&*\\
        {0}&{0}&{*}&{*}&*\\
        {0}&{*}&{*}&{*}&*\\
        {*}&{*}&{*}&{*}&*\\
        {*}&{*}&{*}&{*}&*\end{smallmatrix}\right),\left(\begin{smallmatrix}
        {0}&{0}&{*}&{*}&*\\
        {0}&{*}&{*}&{*}&*\\
        {*}&{*}&{*}&{*}&*\\
        {*}&{*}&{*}&{*}&*\\
        {*}&{*}&{*}&{*}&*\end{smallmatrix}\right)\right].$$
It follows that if the element $v=(A,B,C,D,E)$ has nonzero
  discriminant, then the only possible point of intersection of the
  quadrics cut out by the $Q_i$'s in the hyperplane section $t_5=0$ is at the point $[1:0:0:0]\in \P^3(\C)$. 
	  By Lemma \ref{lemcondredprelim}(c), $v$ corresponds to the identity element in the Selmer
  group of the Jacobian of $C_v$, and so is not strongly irreducible.
\end{proof}

We are now ready to prove Proposition \ref{propaneq0}:
\vspace{2 ex}

\noindent{\bf Proof of Proposition \ref{propaneq0}:} 
Consider the set 
$$\var:=\{ x_{ij}:x\in\{a,b,c,d,e\}, i,j\in\{1,2,3,4,5\},i<j\}.$$
We consider each $u\in\var$ as a $\Z$-valued (resp.\ $\R$-valued)
function on $V_\Z$ (resp.\ $V_\R$) in the obvious way.
Each variable $u\in\var$ has a weight $w(u)$ defined by
$a(s_1,\ldots,s_8)\cdot u = w(u)u$.  The weight $w(u)$ is evidently a
rational function in $s_1,\ldots,s_8$.  We may define a natural partial
order $\lesssim$ on $\var$, where $x_{ij}\lesssim y_{k\ell}$ if $x$ is
either lexicographically ahead of \,or is the same as $y$, $\,i\leq
k$, and $j\leq \ell$. Note that for any $u_1,u_2\in\var$, we have
$u_1\lesssim u_2$ if and only if the exponent of every $s_i$ in
$w(u_1)$ is less than or equal to the corresponding exponent in
$w(u_2)$.

Now given a subset $\ZZ\subset \var$, we define the set $V_\Z(\ZZ)\subset V_\Z$ to
be the set of all elements $v\in V_\Z$ such that $u(v)=0$ for each
$u\in \ZZ$ and $u(v)\neq 0$ for each $u\in \NN=\NN(\ZZ)$, where $\NN(\ZZ)$ is the
set of minimal elements (under the above partial ordering) in
$\var\backslash \ZZ$.
Identically to the proof of \cite[Lemma 11]{dodpf}, we partition $V_\Z$ into disjoint subsets of $V_\Z$ by
the following process.  Start with $S=V_\Z(\{a_{12}\})$. (Thus, $S$ is the
set of all $(A,B,C,D,E)\in V_\Z$ with $a_{12}=0$ and
$a_{13},b_{12}\neq 0$.) At every step of the process, for each set
$V_\Z(\ZZ)$ generated in the previous step and each $u\in \NN$, we
add the set $V_\Z(\ZZ\cup \{u\})$ to our list of subsets,
provided that $\ZZ\cup \{u\}$ does not contain any of the thirteen sets listed in
Proposition~\ref{propreducibility}.

Let $\SS$ denote the set of all subsets of $\var$ generated by the above process.
It is clear that every strongly irreducible element with
$a_{12}=0$ is contained in exactly one of the sets in $\SS$. Thus to prove
Proposition \ref{propaneq0}, it suffices to prove the estimate
$N(S;X)=O(X^{499/600})$ for each $S\in\SS$.
Equation (\ref{eqavgfinal}) implies that given a fixed set $S=V_\Z(\ZZ)\in \SS$, we have the estimate
$$
N(V_\Z(\ZZ);X)=O\left(\int_{s_1,\ldots,s_8=c}^\infty\sigma(\ZZ,a)d^\ast a\right),
$$
where $d^\ast a$ is given by (\ref{dadn}) and $\sigma(\ZZ,a)$ is the number of integer points in the region 
$$B^\pm(0,a;X;\ZZ):=\{v\in B^\pm(0,a;X):u(v)=0\mbox{ for }u\in \ZZ,\mbox{ and }|u(v)|\geq 1\mbox{ for }u\in \NN\}.$$
An element $v\in B^\pm(0,a;X;\ZZ)$ satisfies $u(v)\ll X^{1/60}w(u)$ for
each $u\in \var$, and therefore $\sigma(\ZZ,a)$ is nonzero only if
$X^{1/60}w(u)\gg 1$ for each $u\in \NN$. Since $\NN$ was chosen to be the set of minimal elements in $\var\backslash \ZZ$, it follows that $\sigma(\ZZ,a)$ is nonzero only if
$X^{1/60}w(u)\gg 1$ for every $u\in \var\backslash \ZZ$.

If we define the weight $w(u_1^{r_1}\cdots u_k^{r_k})$ to be
$w(u_1)^{r_1}\cdots w(u_k)^{r_k}$ for $u_1,\ldots,u_k\in\var$ and
$r_1,\ldots,r_k\in\R$, then we have the estimate
\begin{equation}\label{estvzzz}
  \begin{array}{rcl}
N(V_\Z(\ZZ);X)&\!\!\!\!=\!\!\!\!&O\left(\displaystyle\int_{s_1,\ldots,s_8=c}^\infty X^{\textstyle\frac{50-\# \ZZ}{60}}\cdot w\left(\displaystyle\prod_{u\in\var\backslash \ZZ}u\right)d^\ast a\right)\\[0.3in]
&\!\!\!\!=\!\!\!\!&O\left(\displaystyle\int_{s_1,\ldots,s_8=c}^\infty X^{\textstyle\frac{50-\# \ZZ+\rm{deg}(\pi_\ZZ)}{60}}\cdot w\left(\pi_\ZZ\cdot\displaystyle\prod_{u\in\var\backslash \ZZ}u\right)d^\ast a\right),
  \end{array}
\end{equation}
for any $\pi_\ZZ=\prod u_i^{r_i}$ such that all the $u_i$'s are in
$\var\backslash \ZZ$ and all the $r_i$'s are positive real numbers.
The first equality in \eqref{estvzzz} follows by applying
Proposition \ref{davlem} on the set $B^\pm(0,a;X;\ZZ)$, along with the
fact that $\sigma(\ZZ,a)$ is nonzero only if $X^{1/60}w(v)\gg 1$ for
each $u\in \var\backslash \ZZ$. The second equality also follows directly 
from the latter fact. Therefore, if we find $\pi_\ZZ$ as above such that
the exponent of each $s_i$ in
\begin{equation}\label{eqpifactor}
s_1^{-20}s_2^{-30}s_3^{-30}s_4^{-20}s_5^{-20}s_6^{-30}s_7^{-30}s_8^{-20}\cdot
w\Bigl(\pi_\ZZ\cdot\prod_{u\in\var\backslash \ZZ}u\Bigr)
\end{equation}
is negative, then we may conclude that
\begin{equation}\label{estvzzzsec}
N(V_\Z(\ZZ);X)=O\Bigl(X^{\textstyle{\frac{50-\#\ZZ+\rm{deg}(\pi_\ZZ)}{60}}}\Bigr).
\end{equation}

We now describe how, if we have a suitable factor $\pi_\ZZ$ for the
set $\ZZ$ in (\ref{estvzzz})--(\ref{estvzzzsec}), 
then we can obtain the estimate on the right hand side of
\eqref{estvzzzsec} also for $N(V_\Z(\ZZ');X)$ for all 
subsets $\ZZ'\subset \ZZ$.  Indeed, if we can construct
products $\pi_{\ZZ,u}$ of nonnegative powers of elements in
$\var\backslash \ZZ$, for each $u\in \ZZ\backslash\{a_{12}\}$, such
that the exponent of every $s_i$ in $w({u}/{\pi_{\ZZ,u}})$ is negative
and $\pi_\ZZ/\prod_{u\in\ZZ\backslash\{a_{12}\}}\pi_{\ZZ,u}$ is also a
product of nonnegative powers of elements in $\var\backslash \ZZ$,
then we may similarly conclude that
$$N(V_\Z(\ZZ');X)=O\Bigl(X^{\textstyle{\frac{50-\#\ZZ+\rm{deg}(\pi_\ZZ)}{60}}}\Bigr)$$ 
for every subset $\ZZ'$ of $\ZZ$ with $a_{12}\in \ZZ'$. 
This follows by using
$\pi_{\ZZ'}:={\pi_\ZZ}/{\displaystyle\prod_{u\in \ZZ\backslash
    \ZZ'}\pi_{\ZZ,u}}$ and $\ZZ'$ in~place~of $\pi_\ZZ$ and $\ZZ$,
respectively, in \eqref{estvzzz}.

We now have the following lemma which gives a list of sets $\ZZ$ such that every element of $\SS$ is contained in at least one such $\ZZ$. 
\begin{lemma}\label{lemsetsz}
  Let $V_\Z(\ZZ)$ be an element of $\SS$. Then $\ZZ$ is contained in one of the following sets:
  \begin{itemize}
  \item[{\rm (1)}] 
$\{a_{12},a_{13},a_{14},a_{15},a_{23},a_{24},a_{25},a_{34}\}\cup
\{b_{12},b_{13},b_{14},b_{15},b_{23},b_{24}\}
\cup\{c_{12},c_{13},c_{14},c_{23}\}\cup\{d_{12}\}$
  \item[{\rm (2)}] 
$\{a_{12},a_{13},a_{14},a_{15},a_{23},a_{24},a_{25},a_{34}\}\cup
\{b_{12},b_{13},b_{14},b_{15},b_{23},b_{24}\}
\cup\{c_{12},c_{13},c_{14}\}\cup\{d_{12},d_{13}\}$
  \item[{\rm (3)}] 
$\{a_{12},a_{13},a_{14},a_{15},a_{23},a_{24},a_{25},a_{34}\}\cup
\{b_{12},b_{13},b_{14},b_{15},b_{23},b_{24}\}
\cup\{c_{12},c_{13},c_{23}\}\cup\{d_{12},d_{13}\}$
  \item[{\rm (4)}] 
$\{a_{12},a_{13},a_{14},a_{15},a_{23},a_{24},a_{25},a_{34}\}\cup
\{b_{12},b_{13},b_{14},b_{15},b_{23}\}
\cup\{c_{12},c_{13},c_{14},c_{23}\}\cup\{d_{12},d_{13}\}$
  \item[{\rm (5)}] 
$\{a_{12},a_{13},a_{14},a_{15},a_{23},a_{24},a_{25},a_{34}\}\cup
\{b_{12},b_{13},b_{14},b_{23},b_{24}\}
\cup\{c_{12},c_{13},c_{14},c_{23}\}\cup\{d_{12},d_{13}\}$
  \item[{\rm (6)}] 
$\{a_{12},a_{13},a_{14},a_{15},a_{23},a_{24},a_{25}\}\cup
\{b_{12},b_{13},b_{14},b_{15},b_{23},b_{24}\}
\cup\{c_{12},c_{13},c_{14},c_{23}\}\cup\{d_{12},d_{13}\}$
  \item[{\rm (7)}] 
$\{a_{12},a_{13},a_{14},a_{15},a_{23},a_{24},a_{34}\}\cup
\{b_{12},b_{13},b_{14},b_{15},b_{23},b_{24}\}
\cup\{c_{12},c_{13},c_{14},c_{23}\}\cup\{d_{12},d_{13}\}$

  \item[{\rm (8)}] 
$\{a_{12},a_{13},a_{14},a_{15},a_{23},a_{24},a_{34}\}\cup
\{b_{12},b_{13},b_{14},b_{15},b_{23}\}
\cup\{c_{12},c_{13},c_{14},c_{23}\}\cup\{d_{12},d_{13}\}\cup\{e_{12}\}$
  \item[{\rm (9)}] 
$\{a_{12},a_{13},a_{14},a_{15},a_{23},a_{24},a_{34}\}\cup
\{b_{12},b_{13},b_{14},b_{23},b_{24}\}
\cup\{c_{12},c_{13},c_{14},c_{23}\}\cup\{d_{12},d_{13}\}\cup\{e_{12}\}$
  \item[{\rm (10)}] 
$\{a_{12},a_{13},a_{14},a_{15},a_{23},a_{24},a_{25},a_{34},a_{35}\}\cup
\{b_{12},b_{13},b_{14},b_{15},b_{23},b_{24}\}
\cup\{c_{12},c_{13},c_{14}\}\cup\{d_{12}\}$
  \end{itemize}
\end{lemma}
Lemma~\ref{lemsetsz} follows immediately from Proposition \ref{propreducibility}.

\begin{table}[t]
  \centering \renewcommand{\arraystretch}{1.2}
  \begin{tabular}{|c |c |c |}
\hline
Case&$\pi_\ZZ$&$\#\ZZ-\deg\pi_\ZZ$\\
\hline
1&$a_{35}^{4.12}b_{25}^{1.42}b_{34}^{2.56}c_{15}^{1.28}c_{24}^{2.7}d_{13}^{2.98}e_{12}^{3.84}$&$.1$\\[.05in]
\hline
2&$a_{35}^{4.12}b_{25}^{1.42}b_{34}^{2.56}c_{15}^{1.28}c_{23}^{1.7}d_{14}^{3.98}e_{12}^{3.84}$&$.1$\\[.05in]
\hline
3&$a_{35}^{4.34}b_{25}^{2.17}b_{34}^{1.78}c_{14}^{4.73}d_{23}^{2.17}e_{12}^{3.56}$&$.25$\\[.05in]
\hline
4&$a_{35}^{4.12}b_{24}^{4.28}c_{15}^{2.64}d_{14}^{2.64}d_{23}^{1.32}e_{12}^{3.8}$&$.1$\\[.05in]
\hline
5&$a_{35}^{4.12}b_{15}^{1.64}b_{34}^{1.32}c_{24}^{3.96}d_{14}^{2.64}d_{23}^{1.32}e_{12}^{3.8}$&$.2$\\[.05in]
\hline
6&$a_{25}^{3.12}b_{34}^{3.96}c_{15}^{2.64}c_{24}^{1.32}d_{14}^{2.64}d_{23}^{1.32}e_{12}^{3.8}$&$.2$\\[.05in]
\hline
7&$a_{34}^{4.16}b_{25}^{3.16}c_{15}^{3.44}c_{24}^{.54}d_{14}^{2.26}d_{23}^{1.72}e_{12}^{3.62}$&$.1$\\[.05in]
\hline
8&$a_{25}^{3.12}b_{15}^{1.64}b_{34}^{1.32}c_{24}^{3.96}d_{14}^{2.64}d_{23}^{1.32}e_{13}^{4.8}$&$.2$\\[.05in]
\hline
9&$a_{25}^{1.32}a_{35}^{1.8}b_{24}^{4.28}c_{15}^{2.64}d_{14}^{2.64}d_{23}^{1.32}e_{13}^{4.8}$&$.2$\\[.05in]
\hline
10&$a_{45}^{5.05}b_{25}^{1.7}b_{34}^{2.28}c_{15}^{1.14}c_{23}^{1.84}d_{13}^{3.26}e_{12}^{3.63}$&$.1$\\[.05in]
\hline
  \end{tabular}
\caption{The factors $\pi_\ZZ$ used in the proof of Proposition~\ref{propaneq0}.}
\label{tablecusp1}
\end{table}

\begin{table}[htbp]
  \centering\renewcommand{\arraystretch}{.95}
  \begin{tabular}{|l |l |l |l |l|}
\hline
Case 1&Case 2&Case 3&Case 4&Case 5\\
\hline
$\pi_{a_{13}}=a_{35}^{.02}d_{13}^{.98}$\phantom{$4^{4^4}$\!\!\!\!\!\!\!\!\!\!\!\!\!}
&$\pi_{a_{13}}=c_{23}^{.7}c_{15}^{.28}a_{35}^{.02}$&$\pi_{a_{13}}=c_{14}^{.73}a_{35}^{.27}$&$\pi_{a_{13}}=d_{23}^{.32}d_{14}^{.64}a_{35}^{.04}$\!\!&$\pi_{a_{13}}=a_{35}^{.12}d_{23}^{.32}b_{15}^{.56}$\!\!\\[.05in]

$\pi_{a_{14}}=c_{15}^{.28}b_{34}^{.56}b_{25}^{.16}$&$\pi_{a_{14}}=b_{34}^{.56}b_{25}^{.4}a_{35}^{.04}$&$\pi_{a_{14}}=c_{14}$&$\pi_{a_{14}}=c_{15}^{.64}b_{24}^{.28}a_{35}^{.08}$&$\pi_{a_{14}}=b_{34}^{.32}d_{14}^{.64}b_{15}^{.04}$\\[.05in]

$\pi_{a_{15}}=a_{35}$&$\pi_{a_{15}}=b_{25}$&$\pi_{a_{15}}=a_{35}$&$\pi_{a_{15}}=a_{35}$&$\pi_{a_{15}}=a_{35}$\\[.05in]

$\pi_{a_{23}}=c_{24}^{.7}b_{25}^{.26}a_{35}^{.04}$&$\pi_{a_{23}}=a_{35}$&$\pi_{a_{23}}=b_{25}^{.05}b_{34}^{.78}d_{23}^{.17}$\!\!&$\pi_{a_{23}}=a_{35}$&$\pi_{a_{23}}=a_{35}$\\[.05in]

$\pi_{a_{24}}=a_{35}$&$\pi_{a_{24}}=a_{35}$&$\pi_{a_{24}}=a_{35}$&$\pi_{a_{24}}=b_{24}$&$\pi_{a_{24}}=a_{35}$\\[.05in]

$\pi_{a_{25}}=a_{35}$&$\pi_{a_{25}}=a_{35}$&$\pi_{a_{25}}=a_{35}$&$\pi_{a_{25}}=a_{35}$&$\pi_{a_{25}}=a_{35}$\\[.05in]

$\pi_{a_{34}}=a_{35}$&$\pi_{a_{34}}=a_{35}$&$\pi_{a_{34}}=a_{35}$&$\pi_{a_{34}}=a_{35}$&$\pi_{a_{34}}=b_{34}$\\[.05in]

$\pi_{b_{12}}=e_{12}$&$\pi_{b_{12}}=e_{12}$&$\pi_{b_{12}}=e_{12}$&$\pi_{b_{12}}=e_{12}$&$\pi_{b_{12}}=e_{12}$\\[.05in]

$\pi_{b_{13}}=d_{13}$&$\pi_{b_{13}}=d_{14}^{.98}b_{25}^{.02}$&$\pi_{b_{13}}=c_{14}$&$\pi_{b_{13}}=b_{24}$&$\pi_{b_{13}}=c_{24}^{.96}b_{15}^{.04}$\\[.05in]

$\pi_{b_{14}}=b_{25}$&$\pi_{b_{14}}=d_{14}$&$\pi_{b_{14}}=c_{14}$&$\pi_{b_{14}}=b_{24}$&$\pi_{b_{14}}=b_{15}$\\[.05in]

$\pi_{b_{15}}=c_{15}$&$\pi_{b_{15}}=c_{15}$&$\pi_{b_{15}}=b_{25}$&$\pi_{b_{15}}=c_{15}$&$\pi_{b_{23}}=c_{24}$\\[.05in]

$\pi_{b_{23}}=b_{34}$&$\pi_{b_{23}}=b_{34}$&$\pi_{b_{23}}=b_{25}$&$\pi_{b_{23}}=b_{24}$&$\pi_{b_{24}}=c_{24}$\\[.05in]

$\pi_{b_{24}}=b_{34}$&$\pi_{b_{24}}=b_{34}$&$\pi_{b_{24}}=b_{34}$&$\pi_{c_{12}}=e_{12}$&$\pi_{c_{12}}=e_{12}$\\[.05in]

$\pi_{c_{12}}=e_{12}$&$\pi_{c_{12}}=e_{12}$&$\pi_{c_{12}}=e_{12}$&$\pi_{c_{13}}=c_{15}$&$\pi_{c_{13}}=c_{24}$\\[.05in]

$\pi_{c_{13}}=d_{13}$&$\pi_{c_{13}}=c_{23}$&$\pi_{c_{13}}=c_{14}$&$\pi_{c_{14}}=d_{14}$&$\pi_{c_{14}}=d_{14}$\\[.05in]

$\pi_{c_{14}}=c_{24}$&$\pi_{c_{14}}=d_{14}$&$\pi_{c_{23}}=d_{23}$&$\pi_{c_{23}}=d_{23}$&$\pi_{c_{23}}=d_{23}$\\[.05in]

$\pi_{c_{23}}=c_{24}$&$\pi_{d_{12}}=e_{12}$&$\pi_{d_{12}}=e_{12}$&$\pi_{d_{12}}=e_{12}$&$\pi_{d_{12}}=e_{12}$\\[.05in]

$\pi_{d_{12}}=e_{12}$&$\pi_{d_{13}}=d_{14}$&$\pi_{d_{13}}=d_{23}$&$\pi_{d_{13}}=d_{14}$&$\pi_{d_{13}}=d_{14}$\\[.05in]
\hline
\hline
Case 6&Case 7&Case 8&Case 9&Case 10\\
\hline
$\pi_{a_{13}}=a_{25}^{.08}c_{15}^{.6}c_{24}^{.32}$\phantom{$4^{4^4}$\!\!\!\!\!\!\!\!\!\!\!\!\!}
&$\pi_{a_{13}}=a_{34}$&$\pi_{a_{13}}=b_{15}^{.5}d_{14}^{.5}$&$\pi_{a_{13}}=a_{35}^{.8}a_{25}^{.2}$&$\pi_{a_{13}}=d_{13}$\\[.05in]

$\pi_{a_{14}}=d_{14}^{.64}d_{23}^{.32}a_{25}^{.04}$\!\!&$\pi_{a_{14}}=a_{34}$&$\pi_{a_{14}}=a_{25}$&$\pi_{a_{14}}=a_{25}$&$\pi_{a_{14}}=b_{25}^{.7}b_{34}^{.16}c_{15}^{.14}$\\[.05in]

$\pi_{a_{15}}=a_{25}$&$\pi_{a_{15}}=c_{15}$&$\pi_{a_{15}}=b_{15}$&$\pi_{a_{15}}=c_{15}$&$\pi_{a_{15}}=a_{45}$\\[.05in]

$\pi_{a_{23}}=a_{25}^{.7}b_{25}^{.26}a_{35}^{.04}$&$\pi_{a_{23}}=a_{34}$&$\pi_{a_{23}}=a_{25}$&$\pi_{a_{23}}=b_{24}$&$\pi_{a_{23}}=c_{23}^{.84}b_{34}^{.12}a_{45}^{.04}\!\!$\\[.05in]

$\pi_{a_{24}}=a_{25}$&$\pi_{a_{24}}=a_{34}$&$\pi_{a_{24}}=a_{25}$&$\pi_{a_{24}}=b_{24}$&$\pi_{a_{24}}=a_{45}$\\[.05in]

$\pi_{a_{34}}=b_{34}$&$\pi_{a_{25}}=b_{25}$&$\pi_{a_{34}}=b_{34}$&$\pi_{a_{34}}=a_{35}$&$\pi_{a_{25}}=a_{45}$\\[.05in]

$\pi_{b_{12}}=e_{12}$&$\pi_{b_{12}}=e_{12}$&$\pi_{b_{12}}=e_{13}^{.8}d_{23}^{.2}$&$\pi_{b_{12}}=e_{13}^{.8}b_{24}^{.2}$&$\pi_{a_{34}}=a_{45}$\\[.05in]

$\pi_{b_{13}}=b_{34}^{.96}c_{15}^{.04}$&$\pi_{b_{13}}=d_{23}^{.72}d_{14}^{.26}b_{25}^{.02}$\!\!&$\pi_{b_{13}}=c_{24}^{.9}b_{34}^{.1}$&$\pi_{b_{13}}=c_{15}^{.5}d_{14}^{.5}$&$\pi_{a_{35}}=a_{45}$\\[.05in]

$\pi_{b_{14}}=c_{24}$&$\pi_{b_{14}}=c_{24}^{.54}c_{15}^{.44}b_{25}^{.02}$&$\pi_{b_{14}}=d_{14}$&$\pi_{b_{14}}=b_{24}$&$\pi_{b_{12}}=e_{12}$\\[.05in]

$\pi_{b_{15}}=c_{15}$&$\pi_{b_{15}}=c_{15}$&$\pi_{b_{23}}=c_{24}$&$\pi_{b_{15}}=c_{15}$&$\pi_{b_{13}}=d_{13}$\\[.05in]

$\pi_{b_{23}}=c_{34}$&$\pi_{b_{23}}=b_{25}$&$\pi_{b_{24}}=c_{24}$&$\pi_{b_{23}}=b_{24}$&$\pi_{b_{14}}=b_{34}$\\[.05in]

$\pi_{b_{24}}=b_{34}$&$\pi_{b_{24}}=b_{25}$&$\pi_{c_{12}}=c_{24}$&$\pi_{c_{12}}=d_{14}$&$\pi_{b_{15}}=b_{25}$\\[.05in]

$\pi_{c_{12}}=e_{12}$&$\pi_{c_{12}}=e_{12}$&$\pi_{c_{13}}=e_{13}$&$\pi_{c_{13}}=e_{13}$&$\pi_{b_{23}}=c_{23}$\\[.05in]

$\pi_{c_{13}}=c_{15}$&$\pi_{c_{13}}=c_{15}$&$\pi_{c_{14}}=d_{14}$&$\pi_{c_{14}}=d_{14}$&$\pi_{b_{24}}=b_{34}$\\[.05in]

$\pi_{c_{14}}=d_{14}$&$\pi_{c_{14}}=d_{14}$&$\pi_{c_{23}}=d_{23}$&$\pi_{c_{23}}=d_{23}$&$\pi_{c_{12}}=e_{12}$\\[.05in]

$\pi_{c_{23}}=d_{23}$&$\pi_{c_{23}}=d_{23}$&$\pi_{d_{12}}=e_{13}$&$\pi_{d_{12}}=e_{13}$&$\pi_{c_{13}}=d_{13}$\\[.05in]

$\pi_{d_{12}}=e_{12}$&$\pi_{d_{12}}=e_{12}$&$\pi_{d_{13}}=e_{13}$&$\pi_{d_{13}}=e_{13}$&$\pi_{c_{14}}=c_{15}$\\[.05in]

$\pi_{d_{13}}=d_{14}$&$\pi_{d_{13}}=d_{14}$&$\pi_{e_{12}}=e_{13}$&$\pi_{e_{12}}=e_{13}$&$\pi_{d_{12}}=e_{12}$\\[.05in]
\hline
  \end{tabular}
\caption{The factors $\pi_u=\pi_{\ZZ,u}$ used in the proof of Proposition~\ref{propaneq0}.}
\label{tablecusp2}
\end{table}

For each set $\ZZ$ of Lemma \ref{lemsetsz}, we construct in Table
\ref{tablecusp1}, monomials $\pi_\ZZ=\prod u_i^{r_i}$ with $u_i\in
\var\backslash \ZZ$ and $r_i\geq 0$ such that the exponent of each $s_i$
in \eqref{eqpifactor} is negative. In Table \ref{tablecusp2}, for each such
set~$\ZZ$, we determine $\pi_{\ZZ,u}$'s for $u\in \ZZ\backslash\{a_{12}\}$ such that
each $\pi_{\ZZ,u}$ is a product of nonnegative powers of elements in
$\var\backslash \ZZ$, the exponent of every $s_i$ in $w(u/\pi_{\ZZ,u})$ is
negative, and $\pi_\ZZ/\prod_{u\in\ZZ\backslash\{a_{12}\}}\pi_{\ZZ,u}$ is also a product of nonnegative
powers of elements in $\var\backslash \ZZ$. We conclude that
for every $\ZZ'\subset \ZZ$, we have
$N(V_\Z(\ZZ');X)=O(X^{\frac{50-\#\ZZ+\deg(\pi_\ZZ)}{60}})=O(X^{\frac{50-.1}{60}})$. Since
we have constructed $\pi_\ZZ$ and $\pi_{\ZZ,u}$'s for each set $\ZZ$ of Lemma
\ref{lemsetsz}, Proposition \ref{propaneq0} follows. $\Box$

\vspace{.1in}

We have proven that the number of irreducible elements in the
``cuspidal region'' of the fundamental domain is negligible. The next
proposition  states that the number of reducible elements in the ``main
body'' is also negligible:
\begin{proposition}\label{propred}
Let $V_\Z^\red$ denote the set of elements in $V_\Z$ that are not strongly irreducible. Then
\begin{equation}\label{eqprop22}
\int_{na\in\FF}\#\{v\in B^\pm(n,a;X)\cap V_\Z^\red:a_{12}(v)\neq 0\}dnd^\ast a=o(X^{5/6}).
\end{equation}
\end{proposition}
We defer the proof of Proposition \ref{propred} to \S\ref{s36}.

Therefore, in order to estimate $N(V_\Z;X)$, it suffices to count the
number of (not necessarily strongly irreducible) integral points in
the main body of the fundamental domain. We do this in the following
proposition:
\begin{proposition}\label{propmainterm}
We have
\begin{equation*}
\displaystyle\frac{1}{C_{G_0}}\int_{na\in\FF}\#\{B^\pm(n,a;X)\cap V^\irr_\Z\}dnd^\ast a
=\displaystyle\frac15\Vol(\FF\cdot R^\pm(X))+o(X^{5/6}).
 \end{equation*}
\end{proposition}
\begin{proof}
  The proof of Proposition \ref{propmainterm} is very similar to that
  of \cite[Proposition 12]{dodpf}. If $v\in B^\pm(n,a;X)$, then we
  know that $a_{12}(v)=O(X^{1/60}w(a_{12}))$. Thus, by Propositions
  \ref{propaneq0} and \ref{propred}, we obtain
\begin{equation}\label{proofeq}
\displaystyle\frac{1}{C_{G_0}}\int_{na\in\FF}\#\{B^\pm(n,a;X)\cap V^\irr_\Z\}dnd^\ast a=\displaystyle\frac{1}{C_{G_0}}\int_{\substack{na\in\FF\\X^{1/60}w(a_{12})\gg 1}}\#\{B^\pm(n,a;X)\cap V_\Z\}dnd^\ast a+o(X^{5/6}).
\end{equation}
Since $a_{12}$ has minimal weight, and the projection of
$B^\pm(n,a;X)$ onto $a_{12}$ has length greater than an absolute
positive constant when $X^{1/60}w(a_{12})\gg 1$, Proposition
\ref{davlem} implies that the main term on the right hand side of
\eqref{proofeq} is equal to
\begin{equation}\label{eqproofsmallvol}
\displaystyle\frac{1}{C_{G_0}}\int_{\substack{na\in\FF\\X^{1/60}w(a_{12})\gg
    1}}\left[\Vol(B^\pm(n,a;X))+O\left(\frac{\Vol(B^\pm(n,a;X))}{X^{1/60}w(a_{12})}\right)\right]dnd^\ast
a.
\end{equation}
Since the region $\{nak\in\FF:w(a_{12})\ll X^\epsilon\}$ has volume
$o(1)$ for any $\epsilon<1/60$, \eqref{eqproofsmallvol} is equal to
$$
\displaystyle\frac{1}{C_{G_0}}\int_{na\in\FF}\Vol(B^\pm(n,a;X))dnd^\ast a+o(X^{5/6}).
$$
The proposition follows since
$$
\displaystyle\frac{1}{C_{G_0}}\int_{na\in\FF}\Vol(B^\pm(n,a;X))dnd^\ast a=\displaystyle\frac{1}{C_{G_0}}\int_{h\in G_0}\Vol(\FF h\cdot R^\pm(X))dh,
$$
and the volume of $\FF h\cdot R^\pm(X)$ is independent of $h$.
\end{proof}

Propositions \ref{propaneq0}, \ref{propred}, and \ref{propmainterm}
imply that
$$N(V_\Z^\pm,X)=\frac15\Vol(\FF\cdot R^\pm(X))+o(X^{5/6}).$$ Thus, to prove
Theorem \ref{thsec3main}, it only remains to compute the volume
$\Vol(\FF\cdot R^\pm(X))$.

\subsection{Computing the volume}

Let $dv$ denote the Euclidean measure on $V_\R$ normalized so that
$V_\Z$ has covolume $1$. The sets~$R^\pm$ contain exactly one point
having invariants $I$ and $J$ for every pair $(I,J)\in\R\times\R$
satisfying $\pm\Delta(I,J)>0$. Let
$dIdJ$ be the measure on these sets $R^\pm$. Recall that we defined
$\omega$ to be a differential that generates the rank 1 module of
top-degree left-invariant differential forms of $G$ over~$\Z$.  With these measure
normalizations, we have the following proposition whose proof is
identical to that of \cite[Proposition 2.8]{BS}.
\begin{proposition}\label{propjac}
  For any measurable function $\phi$ on $V_\R$, we have
\begin{equation}\label{Jac}
|\J|\cdot\int_{p_{I,J}\in R^{\pm}}
\int_{h\in G_\R}\phi(h\cdot p_{I,J}))\,\omega(h)\,dI dJ=\int_{G_\R\cdot R^{\pm}}\phi(v)dv=5\int_{\pnr}\phi(v)dv,
\end{equation}
where $\J$ is a nonzero rational constant and $p_{I,J}$ is the point
in $R^\pm$ having invariants $I$ and $J$.
\end{proposition}

We now compute the volume of the multiset $\FF\cdot R^\pm(X)$:
\begin{equation}
  \int_{\FF\cdot R^{\pm}(X)}\!\!\!\!\!dv=
  |\J|\cdot\int_{p_{I,J}\in R^{\pm}(X)}\int_{\FF}\omega(h)\,dI\,dJ=|\J|\cdot\Vol(\FF)\int_{R^{\pm}(X)}dI\,dJ.
\end{equation}
Up to an error of $O(X^{1/2})$, the quantity $\int_{R^{\pm}(X)}dI\,dJ$ is
equal to $N^{\pm}(X)$ (see the proof of \cite[Proposition 2.10]{BS} for details).

We conclude that
\begin{equation}
N(V_{\Z}^{\pm};X)=\frac15|\J|\cdot\Vol(G_\Z\backslash G_\R)N^\pm(X)+o(X^{5/6}).
\end{equation}

\subsection{Congruence conditions and a squarefree sieve}\label{congse}

In this subsection, we prove a version of Theorem \ref{thsec3main} where we
count strongly irreducible $G_\Z$-orbits on points $V_\Z$ that satisfy any specified finite set of
congruence conditions.

For any set $S$ in $V_\Z$ that is
definable by congruence conditions, denote by $\mu_p(S)$
the $p$-adic density of the $p$-adic closure of $S$ in $V_{\Z_p}$,
where we normalize the additive measure $\mu_p$ on $V_{\Z_p}$ so that
$\mu_p(V_{\Z_p})=1$.
We then have the following theorem whose proof is identical to that of \cite[Theorem~2.11]{BS}.
\begin{theorem}\label{cong2}
Suppose $S$ is a subset of $\pnv$ defined by finitely many
congruence conditions. Then we have 
\begin{equation}\label{ramanujan}
N(S\cap\pnv;X)
  = N(\pnv;X)
  \prod_{p} \mu_p(S)+o(X^{5/6}),
\end{equation}
where $\mu_p(S)$ denotes the $p$-adic density of $S$ in $V_\Z$, and
where the implied constant in $o(X^{5/6})$ depends only on $S$.
\end{theorem}
We furthermore have the following weighted version of Theorem
\ref{cong2} whose proof is identical to that of~\cite[Theorem 2.12]{BS}.

\begin{theorem}\label{cong3}
  Let $p_1,\ldots,p_k$ be distinct prime numbers. For $j=1,\ldots,k$,
  let $\phi_{p_j}:V_\Z\to\R$ be a $G_\Z$-invariant function on
  $V_\Z$ such that $\phi_{p_j}(v)$ depends only on the congruence
  class of $v$ modulo some power $p_j^{a_j}$ of $p_j$.  Let
  $N_\phi(V_\Z^\pm;X)$ denote the number of irreducible
  $G_\Z$-orbits in $V_\Z^{\pm}$ having height less than $X$,
  where each orbit $G_\Z\cdot v$ is counted with weight
  $\phi(v):=\prod_{j=1}^k\phi_{p_j}(v)$. Then we have
\begin{equation}
N_\phi(V_\Z^{\pm};X)
  = N(V_\Z^{\pm};X)
  \prod_{j=1}^k \int_{v\in V_{\Z_{p_j}}}\tilde{\phi}_{p_j}(v)\,dv+o(X^{5/6}),
\end{equation}
where $\tilde{\phi}_{p_j}$ is the natural extension of ${\phi}_{p_j}$
to $V_{\Z_{p_j}}$ by continuity, $dv$ denotes the additive
measure on~$V_{\Z_{p_j}}$ normalized so that $\int_{v\in
  V_{\Z_{p_j}}}dv=1$, and where the implied constant in the error term
depends only on the local weight functions ${\phi}_{p_j}$.
\end{theorem}

For our applications, we also require a version of Theorem \ref{cong3} which counts 
certain weighted $G_\Z$-orbits where the weights are defined by congruence 
conditions modulo infinitely many prime powers. To describe which weights 
are permissible, we have the following definitions.

A function $\phi:V_\Z\to[0,1]\subset\R$ is said to be {\it defined by congruence
  conditions} if, for all primes $p$, there exist functions
$\phi_p:V_{\Z_p}\to[0,1]$ satisfying the following conditions:
\begin{itemize}
\item[(2)] For all $v\in V_\Z$, the product $\prod_p\phi_p(v)$ converges to $\phi(v)$.
\item[(3)] For each prime $p$, the function $\phi_p$ is
locally constant outside some closed set $S_p \subset V_{\Z_p}$ of measure zero.
\end{itemize}
Such a function $\phi$ is called {\it acceptable} if, for sufficiently
large primes $p$, we have $\phi_{p}(v)=1$ whenever~$p^2\nmid\Delta(v)$.

The key ingredient in proving the stronger version of Theorem \ref{cong3} is the following uniformity/tail estimate:
\begin{theorem}\label{thunif}
For a prime $p$, let $\W_p$ denote the set of elements in $V_\Z$ whose discriminants are divisible by $p^2$.
  Let $\epsilon>0$ be fixed. Then we have:
  \begin{equation}\label{equnif}
N\bigl(\displaystyle\cup_{p>M}\W_p,X\bigr)=O_\epsilon(X^{5/6}/(M\log M)+X^{49/60})+O(\epsilon X^{5/6}),
  \end{equation}
where the implied constant is independent of $M$ and $X$.
\end{theorem}
\begin{proof}
Let $\W^{(1)}_p\subset V_\Z$ be the $G_\Z$-invariant subset consisting
of elements whose discriminants are strongly divisible by $p^2$, where
an element $v$ is said to have discriminant {\it strongly divisible by
  $p^2$} if for every $w\in V_\Z$, we have $p^2\mid \Delta(v+pw)$. For
$\epsilon>0$, let $\FF^{(\epsilon)}\subset\FF$ denote the subset of
elements $na(s_1,s_2,s_3,s_4,s_5,s_6,s_7,s_8)k\in\FF$ such that the
$s_i$ are bounded above by a constant to ensure that
$\Vol(\FF^{(\epsilon)})=(1-\epsilon)\Vol(\FF)$. Then
$\FF^{(\epsilon)}\cdot R^\pm(X)$ is a bounded domain in $V_\R$ that
expands homogeneously as $X$ grows. From \cite[Theorem 3.3]{geosieve}, we
obtain
\begin{equation}\label{unifest1}
\#\{\FF^{(\epsilon)}\cdot R^\pm(X)\bigcap (\cup_{p>M}\W^{(1)}_p)\}=O(X^{5/6}/(M\log M)+X^{49/60}).
\end{equation}
Also, the results of \S\ref{s31}--\S\ref{s33} imply that
\begin{equation}\label{unifest2}
\#\{(\FF\backslash\FF^{(\epsilon)})\cdot R^\pm(X)\bigcap
V_\Z^\irr\}=O(\epsilon X^{5/6}).
\end{equation}
Combining the estimates \eqref{unifest1} and \eqref{unifest2} yields
\eqref{equnif} with $\W_p$ replaced with $\W^{(1)}_p$.

Therefore, it remains to prove \eqref{equnif} with $\W_p$ replaced with $\W_p^{(2)}:=\W_p\setminus
\W^{(1)}_p$.  The set $\W_p^{(2)}$
consists of the elements $v\in V_\Z$ having
discriminant {\it weakly divisible} by~$p^2$, i.e., $p^2$ divides $\Delta(v)$ but 
does not strongly divide $\Delta(v)$. Thus an element $v$ has discriminant
weakly divisible by~$p^2$ precisely when~$p^2\mid \Delta(v)$ and
the genus one
curve over~$\F_p$ corresponding to the reduction of $v$ modulo $p$ has
a single nodal singularity.

Let $v=(A,B,C,D,E)\in\W_p^{(2)}$ be any such element, let
$\bar{v}=(\bar{A},\bar{B},\bar{C},\bar{D},\bar{E})\in V_{\F_p}$ be its
reduction modulo $p$, and let $C$ be the curve over $\F_p$ corresponding to
$\bar{v}$. We may assume that the nodal singularity of $C$ is at
$[1:0:0:0:0]\in\P^4_{\F_p}$, which implies that $\bar{A}$ has rank
$2$. Therefore, by replacing~$\bar{v}$ with a $G_{\F_p}$-translate if
necessary, we may assume that $\bar{a}_{12}$ is the only nonzero
coefficient of $\bar{A}$. We next claim that we may replace $v$ with a
$G_{\F_p}$-translate to ensure that
$\bar{b}_{45}=\bar{c}_{45}=\bar{d}_{45}=\bar{e}_{45}=0$. Indeed, since
$C$ has a double point at $P=[1:0:0:0:0]$, the intersection of $C$ and
the hyperplane section $t_1=0\subset \P^4_{\F_p}$ contains $P$ with
multiplicity at least $2$. As explained in \cite{geosieve}, this implies that
by replacing $\bar{v}$ with a $G_{\F_p}$-translate, if necessary, the
$\F_p$-span of the four $3\times 3$ matrices
$$
\left(\begin{smallmatrix}{0}&{b_{34}}&{b_{35}}&\\{-b_{34}}&{0}&{b_{45}}\\{-b_{35}}&{-b_{45}}&{0}\end{smallmatrix}\right),
\left(\begin{smallmatrix}{0}&{c_{34}}&{c_{35}}&\\{-c_{34}}&{0}&{c_{45}}\\{-c_{35}}&{-c_{45}}&{0}\end{smallmatrix}\right),
\left(\begin{smallmatrix}{0}&{d_{34}}&{d_{35}}&\\{-d_{34}}&{0}&{d_{45}}\\{-d_{35}}&{-d_{45}}&{0}\end{smallmatrix}\right),
\left(\begin{smallmatrix}{0}&{e_{34}}&{e_{35}}&\\{-e_{34}}&{0}&{e_{45}}\\{-e_{35}}&{-e_{45}}&{0}\end{smallmatrix}\right)
$$ has rank at most $2$; thus, by again replacing $\bar{v}$ by a suitable
$G_{\F_p}$-translate, we may assume that
$\bar{b}_{45}=\bar{c}_{45}=\bar{d}_{45}=\bar{e}_{45}=0$.

Let $Z\subset \var$ denote the set $\{a_{ij}:(i,j)\neq (1,2)\}\cup\{b_{45},c_{45},d_{45},e_{45}\}$. Given $v\in \W_p^{(2)}$, we
have already proven that there exists $v'$ in the $G_\Z$-orbit of $v$
such that $p\mid u(v')$ for every $u\in Z$. By evaluating the discriminant polynomial on such a $v'$, we conclude that if $p^2\mid \Delta(v')$, then $p^2\mid a_{45}(v')$.

Let $\gamma\in G_\Q$ be
$$
\gamma:=
\left[\left(\begin{smallmatrix}
{1}& {}&{}&{}&{}\\ {} & {1}&{}&{}&{}\\
{} & {}&{1}&{}&{}\\
{} & {}&{}&{p^{-1}}&{}\\
{}&{} & {}&{}&{p^{-1}}
\end{smallmatrix}\right),
\left(\begin{smallmatrix}
    {1}&{}&{}&{}&{}\\ {}&{p}&{}&{}&{}\\
{}&{}&{p}&{}&{}\\
{}&{}&{}&{p}&{}\\
{}&{}&{}&{}&{p}
\end{smallmatrix}\right)\right].
$$
Then $\gamma\cdot v'$ is an element of $\W_p^{(1)}$, and it has the
same discriminant as $v'$.
We thus obtain a map
$\phi:G_\Z\backslash\W_p^{(2)}\to G_\Z\backslash\W_p^{(1)}$ that is
discriminant-preserving. We now have the following lemma:
\begin{lemma}
The map $\phi$ is at most $2$ to $1$.
\end{lemma}
\begin{proof}
Consider a $G_\Z$-orbit in the image of $\phi$ and an element
$v\in\W_p^{(1)}$ in this orbit of the form $\gamma\cdot v'$ for some $v'\in W_p^{(2)}$. Let
$Q_1(t_1,t_2,t_3,t_4,t_5),\ldots, Q_5(t_1,t_2,t_3,t_4,t_5)$ be the
five quadratic forms corresponding to $v$. It is easy to check that
the action of $\gamma_p^{-1}$ acts on the quadratic forms as follows:
the forms $Q_1$, $Q_2$, and $Q_3$ are multiplied by $p^2$, the forms
$Q_4$ and $Q_5$ are multiplied by $p$, and the variables
$t_2,\ldots,t_5$ are divided by $p$. Thus, for
$\gamma_p^{-1}\cdot v$ to be integral, it is necessary and sufficient that the bottom
right $4\times 4$ submatrices of $Q_4$ and $Q_5$ be multiples of $p$.

Therefore, preimages of $G_\Z\cdot v$ under $\phi$ correspond to
$2$-dimensional subspaces of quinary quadratic forms over $\F_p$ generated by the
reductions modulo $p$ of $Q_1,\ldots,Q_5$, such that the quadratic
forms in this subspace contain a common $4$-dimensional isotropic hyperplane. 
Let $\overline{Q}_1,\ldots,\overline{Q}_5$ denote the reductions modulo $p$ of $Q_1,\ldots,Q_5$. 
Then we claim that no nonzero element
outside the $\F_p$-span of $\overline{Q}_4$ and $\overline{Q}_5$ can have a
$4$-dimensional isotropic subspace. Indeed, if there was such an
element, then we could assume without loss of generality that it was
$\overline{Q}_3$. However, then the
action of $\gamma_p^{-1}$ would take $Q_3$ to $Q_3'$, whose reduction
modulo $p$ would also contain a $4$-dimensional isotropic
subspace. This would force $\gamma_p^{-1}\cdot v\in \W_p^{(1)}$,
contradicting our assumption that
$\gamma_p^{-1}\cdot v\in\W_p^{(2)}$.

Thus, the lemma is true unless $\overline{Q}_4$ and $\overline{Q}_5$
possess more than two common isotropic $4$-dimensional subspaces. This only
happens when $\overline{Q}_4=\overline{Q}_5=0$. In this case, the
reduction modulo~$p$ of the quadratic forms corresponding to
$\gamma_p^{-1}\cdot v$ have $t_0^2$-, $t_0t_1$-, $t_0t_2$-, $t_0t_3$-,
$t_0t_4$-, and $t_0t_5$-coefficients equal to zero.  This again forces 
$\gamma_p^{-1}\cdot v\in\W_p^{(1)}$, a contradiction.
\end{proof}

\noindent

Therefore
$$
N\bigl(\displaystyle\cup_{p>M}\W_p^{(2)},X\bigr)\leq 2N\bigl(\displaystyle\cup_{p>M}\W_p^{(1)},X\bigr)=O_\epsilon(X^{5/6}/(M\log M)+X^{49/60})+O(\epsilon X^{5/6}),
$$
which concludes the proof of Theorem~\ref{thunif}.
\end{proof}

We thus obtain the following theorem:
\begin{theorem}\label{thsqfree}
    Let $\phi:V_\Z\to[0,1]$ be an acceptable function that is defined by
  congruence conditions via the local functions $\phi_{p}:V_{\Z_p}\to[0,1]$. Then, with
  notation as in Theorem~$\ref{cong3}$, we have:
\begin{equation}
N_\phi(V_\Z^\pm;X)
  = N(V_\Z^\pm;X)
  \prod_{p} \int_{v\in V_{\Z_{p}}}\phi_{p}(v)\,dv+o(X^{5/6}).
\end{equation}
\end{theorem}
Theorem \ref{thsqfree} follows from Theorem \ref{thunif} just as \cite[Theorem 2.21]{BS} followed from \cite[Theorem 2.13]{BS}.

\subsection{The number of reducible points and points with large stabilizers in
  the main bodies of the fundamental domains is negligible}\label{s36}

In this section we first prove Proposition \ref{propred}, which states
that the number of integral elements of bounded height that are not strongly irreducible
in the main body of the fundamental domain is negligible. We then also
prove, by similar methods, that the number of strongly irreducible $G_\Z$-orbits of
elements of bounded height having a nontrivial stabilizer in $G_\Q$ 
is negligible.

\vspace{.1in}
\noindent{\bf Proof of Proposition \ref{propred}:} Let $v\in V_\Z$
have invariants $I$ and $J$, and let $p>5$ be a prime. If the
$G_\Q$-orbit of $v$ corresponds to the identity element in the
$5$-Selmer group of $E^{I,J}$, then the $G_{\F_p}$-orbit of the
reduction of $v$ modulo $p$ also corresponds to the identity element of
$E^{I,J}(\F_p)/5E^{I,J}(\F_p)$ under the correspondence of
Theorem~\ref{leme2ep}. Thus, if $\bar{v}\in V_{\F_p}$ is an element
having nonzero discriminant that corresponds to a nontrivial element
in $E^{I(\bar{v}),J(\bar{v})}(\F_p)/5E^{I(\bar{v}),J(\bar{v})}(\F_p)$,
then every $v\in V_\Z$ that reduces to $\bar{v}$ modulo $p$ is
strongly irreducible. Denote the set of all such $\bar{v}\in V_{\F_p}$
by $V_{\F_p}^{\neq \rm{id}}$. We now show that $\#V_{\F_p}^{\neq
  \rm{id}}\gg \#V_{\F_p}/p$ where the implied constant is independent
of $p$. Indeed, by work of Deuring \cite{Deu}, there exists an
elliptic curve $E$ over $\F_p$ such that $\#E(\F_p)$ is a multiple of
$5$. Thus, $E(\F_p)/5E(\F_p)$ is nontrivial, and the nontrivial
elements correspond to elements in $V_{\F_p}^{\neq \rm{id}}$. Next,
note that the set $V_{\F_p}^{\neq \rm{id}}$ is closed under
multiplication by nonzero elements of $\F_p$ and under the action of~$G_{\F_p}$. Therefore, we have $\#V_{\F_p}^{\neq \rm{id}}\gg
p\#G_{\F_p}\gg\#V_{\F_p}/p$. It follows that for any $Y>0$, we have
\begin{equation}
\begin{array}{rcl}
\displaystyle\int_{na\in\FF}\#\{v\in B^\pm(n,a;X)\cap
V_\Z^\red:a_{12}(v)\neq 0\}dnd^\ast
a&=&O\Bigl(X^{5/6}\displaystyle\prod_{p<Y}\bigl(1-\frac{\#V_{\F_p}^{\neq
    \rm{id}}}{\#V_{\F_p}}\bigr)\Bigr)\\[.2in]&=&O\Bigl(X^{5/6}\displaystyle\prod_{p<Y}\bigl(1-\frac1p\bigr)\Bigr).
\end{array}
\end{equation}
The proposition now follows by letting $Y$ tend to infinity. $\Box$

\begin{lemma}\label{lemtotred}
  Let $V_\Z^{{\rm bigstab}}\subset V_\Z$ be the set of elements that are
  strongly irreducible and have a nontrivial stabilizer in $G_\Q$.
  Then we have
$$N(V_\Z^{{\rm bigstab}};X)=o(X^{5/6}).$$
\end{lemma}
\begin{proof}
First, we note that by Proposition~\ref{propaneq0}, it suffices to prove
the estimate (\ref{eqprop22}) with $V_\Z^\red$ replaced by
$V_\Z^{\rm bigstab}$. 

If $v\in V_\Z$ has a nontrivial stabilizer in $G_\Q$, then we see from
Theorem~\ref{leme2ep} that $E^{I(v),J(v)}(\Q)[5]$ must be
nontrivial. If furthermore $E^{I(v),J(v)}$ has good reduction at
$p>5$, then it follows by \cite[\S VII, Proposition~3.1]{Sil} that 
$E^{I(v),J(v)}(\F_p)[5]$ must also be nontrivial. 
Therefore, if $\bar{v}\in V_{\F_p}$ is an element having
nonzero discriminant such that $E^{I(\bar{v}),J(\bar{v})}(\F_p)[5]$ is
trivial, then any strongly irreducible $v\in V_\Z$ that reduces to
$\bar{v}$ modulo $p$ must have trivial stabilizer in $G_\Q$. Denote
the set of all such $\bar{v}\in V_{\F_p}$ by
$V_{\F_p}^{\rm{smallstab}}$. 

The set $V_{\F_p}^{\rm{smallstab}}$ is nonempty because there exists
an elliptic curve $E$ over $\F_p$ such that $\#E(\F_p)$ is prime to
$5$, again by \cite{Deu}; the identity element of
$E(\F_p)/5E(\F_p)$ then corresponds to an element in
$V_{\F_p}^{\rm{smallstab}}$.
The rest of the proof now proceeds identically to that of
Proposition \ref{propred}.
\end{proof}

\section{The average number of elements in the $5$-Selmer groups of elliptic curves}

Let $E$ be an elliptic curve over $\Q$, and define the
invariants $I(E)$ and $J(E)$ of $E$ as in~(\ref{eqEIJ}).
Throughout this section, we work with the slightly different height $H'$ on
elliptic curves $E$, defined by
\begin{equation}\label{eqEH}
  H'(E):=\max(|I(E)|^3,J(E)^2/4),
\end{equation}
so that the height on elliptic curves agrees with the height on $V_\Z$ defined in
\eqref{heightvz}. Note that since the heights $H$ and $H'$ on elliptic
curves differ only by a
constant factor, they induce the same ordering on the set of all
(isomorphism classes of) elliptic curves over $\Q$.  

In this section, we prove Theorem \ref{mainellip}
by averaging the size of the $5$-Selmer group of all elliptic curves
over $\Q$, when these curves are ordered by height. In fact, we prove a
generalization of these theorems that allows us to impose certain
infinite sets of congruence conditions on the defining equations of
the elliptic curves. To state this more general theorem, we need the
following definitions.

For each prime $p$, let $\Sigma_p$ be a closed subset of
$\Z_p^2\backslash\{\Delta=0\}$ whose boundary has measure~$0$. To this
collection $(\Sigma_p)_p$, we associate the family $F_\Sigma$ of elliptic curves,
such that $E^{I,J}\in F_\Sigma$ whenever $(I,J)\in \Sigma_p$ for all
$p$. Such a family of elliptic curves over $\Q$ is said to be {\it
  defined by congruence conditions}. We may also impose ``congruence
conditions at infinity'' on $F_\Sigma$ by insisting that an elliptic
curve $E^{I,J}$ belongs to $F_\Sigma$ if and only if $(I,J)$ belongs
to $\Sigma_\infty$, where $\Sigma_\infty$ is equal to
$\{(I,J)\in\R^2:\Delta(I,J)>0\}$, $\{(I,J)\in\R^2:\Delta(I,J)<0\}$, or
$\{(I,J)\in\R^2:\Delta(I,J)\neq 0\}$.

For a family $F$ of elliptic curves defined by congruence conditions,
let $\Inv(F)$ denote the set $\{(I,J)\in\Z\times\Z:E^{I,J}\in F\}$,
and $\Inv_p(F)$ the $p$-adic closure of $\Inv(F)$ in
$\Z_p^2\backslash\{\Delta=0\}$. We define $\Inv_\infty(F)$ 
to be 
$\{(I,J)\in\R^2:\Delta(I,J)>0\}$,
$\{(I,J)\in\R^2:\Delta(I,J)<0\}$, or $\{(I,J)\in\R^2:\Delta(I,J)\neq
0\}$ in accordance with whether $F$ contains only curves of positive
discriminant, negative discriminant, or both, respectively.
Such a family $F$ of elliptic curves is said to be
{\it large} if, for all but finitely many primes $p$, the set
$\Inv_p(F)$ contains at least those pairs $(I,J)\in\Z_p\times\Z_p$
such that $p^2\nmid\Delta(I,J)$. Our purpose in this
section is to prove the following theorem, which extends Theorem~\ref{mainellip} to
more general congruence families of elliptic curves:

\begin{theorem}\label{ellipall}
  Let $F$ be a large family of elliptic curves. When elliptic
  curves $E$ in $F$ are ordered by height, the average size of the
  $5$-Selmer group $S_5(E)$ is equal to $6$.
\end{theorem}

\subsection{Assigning weights to elements in $V_\Z$, and a local mass computation}
Let $F$ be a fixed large family of elliptic curves.
Recall that for an elliptic curve $E^{I,J}$, 
Proposition~\ref{propselparz} asserts that nontrivial elements in $S_5(E^{I,J})$
are in bijection with $G_\Q$-equivalence classes on the set of locally
soluble and strongly irreducible elements in $V_\Z$ having invariants $I$
and $J$. In order to use the counting results of Section~3 to prove Theorem
\ref{ellipall}, we need to define an appropriate weight function on $V_\Z$.

For $v\in V_\Z$, (resp.\ $v\in V_{\Z_p}$), let $B(v)$
(resp.\ $B_p(v)$) denote a set of representatives for the action of
$G_\Z$ (resp.\ $G_{\Z_p}$) on the $G_\Q$-equivalence class of $v$ in
$V_\Z$ (resp.\ the $G_{\Q_p}$-equivalence class of $v$ in
$V_{\Z_p}$). We define our weight function $\phi$ via:
\begin{equation}\label{eqmx}
  \phi(v):=
  \begin{cases}
    \Bigl(\displaystyle\sum_{v'\in B(v)}\frac{\#\Aut_\Q(v')}{\#\Aut_\Z(v')}\Bigr)^{-1} &\text{if $v$ is locally soluble and $(I(v),J(v))\in \Inv_p(F)$ for all $p$;}\\[.1in]
    \qquad\qquad 0 &\text{otherwise},
  \end{cases}
\end{equation}
where $\Aut_\Q(v)$ and $\Aut_\Z(v)$ denote the stabilizers of $v\in
V_\Z$ in $G_\Q$ and in $G_\Z$, respectively. Since Lemma \ref{lemtotred}
states that $\Aut_\Q(v)$ is trivial for all but a negligible set of 
elements $v\in V_\Z$, the function $\phi$ also satisfies the following
three conditions at all but a negligible set of $v$:
\begin{enumerate}
\item If $v\in V_\Z$ is not locally soluble, then $\phi(v)$ is zero.
\item If $(I(v),J(v))$ is not in $\Inv(F)$, then $\phi(v)$ is zero.
\item Otherwise, $\phi(v)$ is the reciprocal of the number of
  $G_\Z$-orbits in the $G_\Q$-equivalence class of $v$ in $V_\Z$.
\end{enumerate}

For the application of Theorem \ref{cong3} to counting $G_\Z$-orbits
on $V_\Z$ weighted by $\phi$, we need to define the following local
weight functions $\phi_p:V_{\Z_p}\to\R_{\geq 0}$:
\begin{equation}\label{eqmxp}
  \phi_p(v):=
  \begin{cases}
    \Bigl(\displaystyle\sum_{v'\in B_p(v)}\frac{\#\Aut_{\Q_p}(v')}{\#\Aut_{\Z_p}(v')}\Bigr)^{-1} &\text{if $v$ is $\Q_p$-soluble and $(I(v),J(v))\in \Inv_p(F)$;}\\[.1in]
    \qquad\qquad 0 &\text{otherwise},
  \end{cases}
\end{equation}
where $\Aut_{\Q_p}(v)$ and $\Aut_{\Z_p}(v)$ denote the stabilizer of
$v\in V_{\Z_p}$ in $G_{\Q_p}$ and $G_{\Z_p}$, respectively. We then 
have the following proposition.
\begin{proposition}\label{propprod}
  If $v\in V_\Z$ has nonzero discriminant, then $\phi(v)=\prod_p\phi_p(v)$.
\end{proposition}
Noting the fact that the group $G_\Q$ has class number one, the proof
of the above proposition is identical to that of \cite[Proposition~3.6]{BS}.

We end the subsection with a proposition that evaluates $\int_{V_{\Z_p}}\phi_p(v)dv$.
\begin{proposition}\label{propmasseval}
  We have
  \begin{equation*}
    \begin{array}{rl}
    \displaystyle\int_{v\in V_{\Z_p}}\phi_p(v)dv&=|\J|_p\cdot\Vol(G_{\Z_p})\cdot\displaystyle\int_{(I,J)\in \Inv_p(F)}\displaystyle\frac{\#(E^{I,J}(\Q_p)/5E^{I,J}(\Q_p))}{\#(E^{I,J}(\Q_p)[5])}dv\\[.25in]&=
    \begin{cases}
      \phantom{55\cdot}\!|\J|_p\cdot\Vol(G_{\Z_p})\cdot\Vol(\Inv_p(F))\quad\text{if $p\neq 5$;}\\[.1in]
      5\cdot|\J|_p\cdot\Vol(G_{\Z_p})\cdot\Vol(\Inv_p(F))\quad\text{if $p=5$}.
    \end{cases}
  \end{array}
  \end{equation*}
\end{proposition}
The first equality in Proposition~\ref{propmasseval} follows from an
argument identical to \cite[Proposition 3.9]{BS}. The second follows
from an argument identical to the proof of \cite[Lemma 3.1]{BK}, which
shows that $\#(E^{I,J}(\Q_p)/5E^{I,J}(\Q_p))$ is equal to
$\#(E^{I,J}(\Q_p)[5])$ when~$p\neq 5$ and equal to
$5\#(E^{I,J}(\Q_p)[5])$ when~$p=5$.

\subsection{The proof of Theorem \ref{ellipall}}
Let $F$ be a large family of elliptic curves.  We start with the
following proposition that is proven in \cite[Theorem 3.17]{BS}.
\begin{proposition}\label{thcountellip}
Let $N(F^\pm,X)$ denote the number of
  elliptic curves in $F$ such that $H'(E)\leq X$ and $\Delta(E)\in\R^\pm$. Then
$$
N(F^\pm;X)=N^\pm(X)\displaystyle\prod_{p}\Vol(\Inv_p(F))+o(X^{5/6}).
$$
\end{proposition}

The results of \S 2 and \S 4.1 imply that we have
\begin{equation}
  \sum_{\substack{E\in F^\pm\\H'(E)<X}}(\#S_5(E)-1)=N_{\phi}(V_\Z^\pm,X)+o(X^{5/6}).
\end{equation}
By Propositions~\ref{Fisherprop} and \ref{propprod}, it follows that $\phi$ is
acceptable. Therefore, we may use Theorem \ref{thsqfree} and
Proposition~\ref{propmasseval} to estimate
$N_\phi(V_\Z^\pm;X)$, obtaining
\begin{equation*}
\begin{array}{rcl}
  \displaystyle\lim_{X\to\infty}\displaystyle\frac{\displaystyle\sum_{\substack{E\in F^\pm\\H'(E)<X}}(\#S_5(E)-1)}{\displaystyle\sum_{\substack{E\in F^\pm\\H'(E)<X}}1}
&=&\displaystyle\lim_{X\to\infty}\frac{N(V_\Z^\pm,X)\cdot\displaystyle\prod_p\displaystyle\int_{v\in V_{\Z_p}}\phi_p(v)dv}{N^\pm(X)\cdot\displaystyle\prod_{p}\Vol(\Inv_p(F))}\\
&=&\displaystyle\lim_{X\to\infty}\frac{\frac15|\J|\Vol(G_\Z\backslash G_\R) N^\pm(X)\cdot5\displaystyle\prod_{p}|\J|_p\Vol(G_{\Z_p})\Vol(\Inv_p(F))}{N^\pm(X)\cdot\displaystyle\prod_{p}\Vol(\Inv_p(F))}\\
&=&\tau(G),
\end{array}
\end{equation*}
where $\tau(G)=5$ denotes the Tamagawa number of $G$. We have proven Theorem \ref{ellipall}, and hence also Theorems~\ref{mainellip} and \ref{ellipcong}.

\section{Families of elliptic curves with equidistributed root number}

In this section, our aim is to construct a union $F$ of large families of
elliptic curves in which exactly $50\%$ have root number $1$, and where the density of $F$
among all elliptic curves is large (indeed, $>55\%$).  

Recall that the root number $r(E)$ of an elliptic curve $E$ can be expressed as
a local product $r(E)=-\prod_pr_p(E)$, where $r_p(E)$ is the {\it
  local root number} of $E$ at $p$. 
Local root numbers of elliptic curves over $\Q$ were computed in \cite{Hal} and \cite{Rh}, and these computations will be key in our constructions. 
The local root number $r_p(E)$ of an elliptic curve $E$ having
multiplicative reduction at $p$ is $1$ or $-1$ depending on whether
the reduction of $E$ at $p$ is split or non-split, respectively. When
applying sieve methods, it can often be difficult to distinguish
between these two cases. For example, it is not known whether the sum
$\sum\mu(-4A^3-27B^2)$ of M\"obius function values over all pairs $(A,B)$ having height less than
$X$ is $o(X^{5/6})$, which has been the traditional approach to this
type of problem.\footnote{There has been progress on obtaining
  equidistribution of root numbers of elliptic curves in one-parameter
  families; see Helfgott \cite{Helfgott}.}
To circumvent this issue, we
take the indirect approach of working with $d(E)=\prod_{p}d_p(E)$
instead of $r(E)$, where
$$
\,d(E):=r(E)r(E_{-1}),\,\;
$$
$$
d_p(E):=r_p(E)r_p(E_{-1}).
$$
Here $E_{-1}$ denotes the quadratic twist of $E$ by $-1$.

In the rest of this section, we construct a finite union of large
families $F$ of elliptic curves such that every curve $E\in F$
satisfies $E_{-1}\in F$ and $d(E)=-1$. Since the height of $E$ is
equal to the height of $E_{-1}$, it follows that exactly $50\%$ of
elliptic curves in $F$ have root number $1$.
The definitions of $d(E)$ and $d_p(E)$ imply immediately that 
$d(E)=\prod_pd_p(E)$. Denoting $d_2(E)d_3(E)$ and $\prod_{p>3}d_p(E)$
by $d_6(E)$ and $d_{1/6}(E)$, respectively, it follows that
$d(E)=d_6(E)d_{1/6}(E)$. 

We use $\Delta_p(E)$ and $\Delta_{p'}(E)$ to denote
$p^{\nu_p(\Delta(E))}$ and $\Delta(E)/\Delta_p(E)$, respectively.  We
further denote $\Delta(E)/(\Delta_2(E)\cdot\Delta_3(E))$ by
$\Delta_{6'}(E)$. In the next proposition, we construct two families of
elliptic curves $E$, defined by finitely many congruence conditions
modulo powers of $2$ and $3$, in which we control $d_6(E)$ in terms of
$\Delta_{6'}(E)$.

\begin{proposition}\label{prop23sum}
  There exist two families $F_1$ and $F_2$ of elliptic curves
  $E_{A,B}$, defined by finitely many congruence conditions on $A$ and
  $B$ modulo powers of $2$ and $3$, such that:
\begin{itemize}
 \item[{\rm 1}] $d_6(E)\equiv\Delta_{6'}(E)\!\pmod{4}$ for $E\in F_1$;
  \item[{\rm 2}] $d_6(E)\not\equiv\Delta_{6'}(E)\!\pmod{4}$ for $E\in F_2$;
  \item[{\rm 3}] The density of $F_1$ is greater than $59.179\%$;
  \item[{\rm 4}] The density of $F_2$ is greater than $40.32\%$;
\item[{\rm 5}]  $F_1$ and $F_2$ are closed under twisting by $-1$.
\end{itemize}
\end{proposition}
\begin{proof}
Using the local root number computations at the prime $2$ in \cite[Table
  1]{Hal}, we construct families of elliptic curves $E$ in Table
\ref{table2} with prescribed values of $d_2(E)$ and
$\Delta_{2'}(E)\!\pmod{4}$. Similarly, we use the local root number
computations at the prime $3$ in \cite[Table 2]{Hal} to construct families of
elliptic curves $E$ in Table~\ref{table3} with prescribed values of
$d_3(E)$ and $\Delta_3(E)\!\pmod{4}$. We also compute the densities of
these families.  Tables \ref{table2} and \ref{table3} are to be read
as follows: each row except the last corresponds to a family of
elliptic curves of the form $y^2=x^3+Ax+B$ defined by congruence conditions
modulo powers of 2 and 3, respectively. These families are
disjoint. The first three columns describe the family by specifying
the condition that $A$ and $B$ must satisfy. The fourth column gives
the density of the family. The final four columns give the relative
density of elliptic curves $E$ within the family with prescribed
values of $\Delta(E)\!\pmod{4}$ and $d_p(E)$. In the final row we simply
sum the densities over all the other rows. For example, the first
row of Table \ref{table2} corresponds to the family of elliptic curves
satisfying $2\nmid A$ and $2\nmid B$. This family has density $1/4$
among all integer pairs $(A,B)$. The prime to $2$ part of the
discriminant of such elliptic curves is always $1$ modulo $4$ and
exactly three quarters of such elliptic curves satisfy
$d_2(E)=1$. Hence the final four entries of the first row are $3/4$,
$1/4$, $0$, and $0$. All the rows (apart from the last one) in both
tables can be read similarly.

For $(i,j)\in \{(1,1),(1,-1),(3,1),(3,-1)\}$, let $G_2(i,j)$ denote
the family of elliptic curves $E$ in Table \ref{table2} with
$(\Delta_{2'}(E)\!\pmod{4},d_2(E))=(i,j)$, and let $G_3(i,j)$ denote the
family in Table \ref{table3} with
$(\Delta_3(E)\!\pmod{4},d_3(E))=(i,j)$. The families $G_p(i,j)$ are
defined by finitely many congruence conditions modulo powers of $p$, and their
densities are listed in the final row of Tables \ref{table2} and
\ref{table3}. 

\begin{table}[t]
  \centering\renewcommand{\arraystretch}{1.2}
  \begin{tabular}{|c | c| c| c|c|c|c|c|}
\hline
$v_2(A)$&$v_2(B)$&Additional&Density&\multicolumn{4}{|c|}{$\mbox{Relative density with given }(\Delta_{2'}\!\!\pmod{4},d_2)$}\\
\cline{5-8}
&&Condition&&$(1,1)$&$(1,-1)$&$(3,1)$&$(3,-1)$\\
\hline
0&0&-&$2^{-2}$&$\frac34$&$\frac14$&0&0\\
\hline
$\geq 1$&0&-&$2^{-2}$&$\frac34$&$\frac14$&0&0\\
\hline
$\geq 1$&1&-&$2^{-3}$&$\frac12$&$\frac12$&0&0\\
\hline
0&$\geq 2$&-&$2^{-3}$&$\frac14$&$\frac14$&$\frac12$&0\\
\hline
1&2&-&$2^{-5}$&$\frac12$&$\frac12$&0&0\\
\hline
2&2&-&$2^{-6}$&$\frac12$&$\frac12$&0&0\\
\hline
$\geq 3$&2&-&$2^{-6}$&$1$&0&0&0\\
\hline
$\geq 1$&3&-&$2^{-5}$&$0$&$\frac34$&0&$\frac14$\\
\hline
$1$&$\geq 4$&-&$2^{-6}$&$\frac12$&0&$\frac12$&0\\
\hline
$3$&$\geq 4$&-&$2^{-8}$&$\frac7{16}$&$\frac7{16}$&$\frac1{16}$&$\frac1{16}$\\
\hline

$\geq 4$&4&-&$2^{-9}$&$0$&$1$&0&$0$\\
\hline

$2$&$\geq 5$&-&$2^{-8}$&$\frac12$&$0$&$\frac14$&$\frac14$\\
\hline
$\geq 4$&5&-&$2^{-10}$&$0$&$1$&0&$0$\\
\hline

0&1&$v_2(\Delta)=7$&$2^{-4}$&$\frac12$&0&0&$\frac12$\\
\hline
0&1&$v_2(\Delta)=8$&$2^{-5}$&$\frac14$&$\frac14$&$\frac12$&$0$\\
\hline
0&1&$v_2(\Delta)=9$&$2^{-6}$&$\frac12$&0&0&$\frac12$\\
\hline
0&1&$v_2(\Delta)=10$&$2^{-7}$&$\frac14$&$\frac14$&$\frac14$&$\frac14$\\
\hline
0&1&$v_2(\Delta)=11$&$2^{-8}$&$\frac14$&$\frac14$&$\frac14$&$\frac14$\\
\hline

2&4&$v_2(\Delta)=13$&$2^{-9}$&$\frac12$&0&0&$\frac12$\\
\hline
2&4&$v_2(\Delta)=14$&$2^{-10}$&$0$&$\frac12$&0&$\frac12$\\
\hline
2&4&$v_2(\Delta)=15$&$2^{-11}$&$\frac12$&0&$\frac12$&0\\
\hline

\multicolumn{3}{|c|}{Total}&$\geq .9946$&$\;\,\geq .5703\;\,$&$\;\geq .2814\;$&$\;\geq .0903\;$&$\geq .0524$\\
\hline

  \end{tabular}
\caption{Densities of elliptic curves having prescribed values of $\Delta_{2'}\pmod{4}$ and $d_2$}
\label{table2}
\end{table}

We now define the families $F_1$ and $F_2$ to be
\begin{equation}
\begin{array}{rcl}
F_1&:=&\displaystyle\bigcup_{\substack{i,j,k,\ell\\ijk\ell\equiv 1\!\!\!\!\!\pmod{4}}}\bigl(G_2(i,j)\cap G_3(k,\ell)\bigr),\\[0.4in]
F_2&:=&\displaystyle\bigcup_{\substack{i,j,k,\ell\\ijk\ell\equiv 3\!\!\!\!\!\pmod{4}}}\bigl(G_2(i,j)\cap G_3(k,\ell)\bigr).
\end{array}
\end{equation}
Since the first two conditions of the proposition are invariant under
twisting by $-1$, the final condition is easily satisfied by replacing
(if necessary) $F_1$ and $F_2$ by $F_1\cup \{E:E_{-1}\in F_1\}$ and
$F_2\cup\{E:E_{-1}\in F_2\}$, respectively. This concludes the proof
of the proposition.
\end{proof}

\begin{table}[t]
  \centering\renewcommand{\arraystretch}{1.2}
  \begin{tabular}{|c | c| c|c| c|c|c|c|}
\hline
$v_3(A)$&$v_3(B)$&Additional&Density&\multicolumn{4}{|c|}{Relative
  density with given $(\Delta_3\!\!\pmod{4},d_3)$}\\
\cline{5-8}
&&Condition&&$(1,1)$&$(1,-1)$&$(3,1)$&$(3,-1)$\\
\hline
0&$\geq 0$&-&$\frac{2}{3}$& 1&0&0&0\\
\hline
$\geq 2$& 0&-&$\frac{2}{3^3}$&0&0&$\frac13$&$\frac23$\\
\hline
$1$  &1&-&$\frac{4}{3^4}$&0&0&0&1\\
\hline
$\geq 2$& 1&-&$\frac{2}{3^4}$&0&0&0&$1$\\
\hline
1 &$\geq 2$&-&$\frac{2}{3^4}$&0&0&1&0\\
\hline
2  &2&-&$\frac{4}{3^6}$&1&0&0&0\\
\hline
$\geq 3$& 2&-&$\frac{2}{3^6}$&0&0&0&1\\
\hline
2  &$\geq 3$&-&$\frac{2}{3^6}$&0&1&0&0\\
\hline
$\geq 4$& 3&-&$\frac{2}{3^8}$&0&0&$\frac{1}{3}$&$\frac{2}{3}$\\
\hline
$3$& 4&-&$\frac{4}{3^9}$&0&0&0&1\\
\hline
$\geq 4$& 4&-&$\frac{2}{3^9}$&0&0&0&1\\
\hline
$3$& $\geq 5$&-&$\frac{2}{3^9}$&0&0&1&0\\
\hline
$4$& 5&-&$\frac{4}{3^{11}}$&1&0&0&0\\
\hline
$\geq 5$& 5&-&$\frac{2}{3^{11}}$&0&0&0&1\\
\hline

1&0&$v_3(\Delta)=3$&$\frac{2}{27}$&0&0&$\frac56$&$\frac16$\\
\hline
1&0&$v_3(\Delta)=4$&$\frac{4}{81}$&1&0&$0$&$0$\\
\hline
1&0&$v_3(\Delta)=5$&$\frac{4}{3^5}$&0&0&$\frac12$&$\frac12$\\
\hline
1&0&$v_3(\Delta)=2n$, $n\geq 3$&$\frac{1}{2\cdot3^4}$&1&0&0&0\\
\hline
1&0&$v_3(\Delta)=2n+1$, $n\geq 3$&$\frac{1}{2\cdot3^5}$&0&0&1&0\\
\hline

\multicolumn{3}{|c|}{Total}&$\geq .9993$&$\;\geq .7277\;$&$\;\geq .0027\;$&$\;\,\geq .1216\;\,$&$\geq .1472$\\
\hline

  \end{tabular}
\caption{Densities of elliptic curves having prescribed values of $\Delta_3\pmod{4}$ and $d_3$.}
\label{table3}
\end{table}

We next prove a result analogous to Proposition \ref{prop23sum}, but where
we now instead control the value of $d_{1/6}$ in terms of the absolute value of
$\Delta_{6'}$. To this end, we have the following lemma.
\begin{lemma}\label{lemdp}
  If $p>3$ is prime, then $d_p(E)=-1$ if and only if $E$ has
  multiplicative reduction at $p$ and $p\equiv 3\pmod{4}$.
\end{lemma}
\begin{proof}
  Since $p$ is odd, we know that the reduction type (good,
  multiplicative, or additive) of $E$ is the same as the reduction
  type of $E_{-1}$.  It follows from \cite[Propositions 2 and 3]{Rh}
  that if $E$ has good or additive reduction at $p$, then
  $r_p(E)=r_p(E_{-1})$.

  Assume that $E$ has multiplicative reduction at $p$. Then, from
  \cite[Proposition 3]{Rh}, we see that $r_p(E)=1$ if and only if the
  reduction of $E$ at $p$ is split. Thus $d_p(E)=1$ exactly when both
  $E$ and $E_{-1}$ have split reduction or when both $E$ and $E_{-1}$
  have nonsplit reduction at $p$. It can be checked that this happens
  precisely when $p\equiv 1\pmod{4}$. The lemma follows.
\end{proof}

\begin{proposition}\label{prop5sum}
There exist two finite unions of large families $F_{3}$ and $F_{4}$ 
of elliptic curves $E_{A,B}$, defined by congruence conditions on $A$ and $B$ modulo primes greater than $3$, such that:
\begin{itemize}
 \item[{\rm 1}] $d_{1/6}(E)\equiv|\Delta_{6'}(E)|\pmod{4}$ for $E\in F_{3}$;
  \item[{\rm 2}] $d_{1/6}(E)\not\equiv|\Delta_{6'}(E)|\pmod{4}$ for $E\in F_{4}$;
  \item[{\rm 3}] The density of $F_{3}$ is greater than $96.689\%$;
  \item[{\rm 4}] The density of $F_{4}$ is greater than $3.26\%$;
\item[{\rm 5}] $F_3$ and $F_4$ are closed under twisting by $-1$.
\end{itemize}
\end{proposition}
\begin{proof}
For an elliptic curve $E=E_{A,B}$ and a prime $p$ greater than $3$, define $\alpha_p(E)$ by
  \begin{equation}
    \begin{array}{rcl}
      \alpha_p(E)=\alpha_p(A,B)&\!\!\!\!:=\!\!\!\!&d_p(E_{A,B})\cdot (-1)^{v_p(\Delta(A,B))}\mbox{ if $p\equiv 3\!\!\!\!\pmod{4}$,}\\[.1in]
      \alpha_p(E)=\alpha_p(A,B)&\!\!\!\!:=\!\!\!\!&1\mbox{ otherwise}.
    \end{array}
  \end{equation}
If $p>3$ is a prime congruent
  to $3$ modulo $4$, then $\alpha_p(E)$ is $1$ if and only if $E$ has
  good reduction at $p$, or $E$ has multiplicative reduction at $p$
  and $v_p(\Delta(E))$ is odd, or $E$ has additive reduction at~$p$
  and $v_p(\Delta(E))$ is even.
We control the quantity $\alpha_p(A,B)$ by imposing one
  of the following two conditions:
  \begin{enumerate}
  \item $E$ has good reduction at $p$, or $E$ has multiplicative
    reduction at $p$ and $v_p(\Delta(E))\in\{1,3\}$, or $E$ has
    additive reduction at $p$ and $v_p(\Delta(E))\in\{2,4\}$. In
    either case, we have $\alpha_p(E)=1$.
  \item $E$ has multiplicative
    reduction at $p$ and $v_p(\Delta(E))=2$ or $E$ has
    additive reduction at $p$ and $v_p(\Delta(E))=3$. In
    both cases, we have $\alpha_p(E)=-1$.
  \end{enumerate}
  Define $F_{3}$ to be the set of elliptic curves $E$ such that $E$
  satisfies the first of the above two conditions at all primes
  congruent to $3$ modulo $4$, or $E$ satisfies the first of the above
  two conditions at all but two primes congruent to $3$ modulo $4$ and
  satisfies the second condition at these two primes. To ensure that
  we are constructing a finite union of large families, we 
  further assume that these two primes are smaller than
  $10000$. Similarly, define $F_{4}$ to be the set of elliptic curves
  that satisfy the first condition at all but one prime (which is
  bounded by 10000) and satisfies the second condition at this one
  prime. These two sets are both clearly finite unions of large
  families. Furthermore, since Lemma~\ref{lemdp} implies that
  $\prod_{p>3}\alpha_p(E)\equiv d_{1/6}(E)\cdot
  |\Delta_{6'}(E)|\!\pmod{4}$, we have
  $d_{1/6}(E)\equiv|\Delta_{6'}(E)|\!\pmod{4}$ if $E\in F_{3}$ and
  $d_{1/6}(E)\not\equiv|\Delta_{6'}(E)|\!\pmod{4}$ if $E\in F_{4}$.

It is easy to compute the densities of $F_3$ and $F_4$. Assume that
$p$ is a prime greater than~$3$.  Then it follows from an elementary
computation that a density of $1-\frac1p$ of elliptic curves have
good reduction at $p$, a density of $\frac1p-\frac1{p^2}$ have
multiplicative reduction at $p$, and a density of $\frac1{p^2}$ have
additive reduction at $p$.

Suppose an elliptic curve $E$ has additive reduction at $p$. Then $E$
is given by an equation $y^2=x^3+ax+b$, where $p$ divides both~$a$ and~$b$. This ensures that $p^2\mid\Delta(E)=4a^3-27b^2$. Clearly
$p^3\mid\Delta(E)$ if and only if $p^2\mid b$. Hence, of elliptic
curves $E$ having additive reduction at $p$, a density of
$\frac{p-1}p$ satisfy $v_p(\Delta(E))=2$. Now if $E:y^2=x^3+ax+b$ has
additive reduction at $p$ with $p^2\mid b$ (so that
$p^3\mid\Delta(E)$), then $p^4\mid\Delta(E)$ if and only if $p^2\mid
a$. Similarly, if we assume that $p^2$ divides both $a$ and~$b$, then
$p^5\mid\Delta(E)$ if and only if $p^3\mid b$. Thus, among elliptic
curves $E$ having additive reduction at $p$, a density of
$\frac{p-1}{p^2}$ satisfy $v_p(\Delta(E))=3$ and a density of
$\frac{p-1}{p^3}$ satisfy $v_p(\Delta(E))=4$.

  Finally, suppose that an elliptic curve $E$ has multiplicative
  reduction at $p$. Then we may assume that $E$ is given by an
  equation $y^2=x^3+cx^2+ax+b$, where $p$ divides both $a$ and~$b$. This ensures that $p\mid\Delta(E)$. As above, we may verify
  that among elliptic curves $E$ having multiplicative reduction at
  $p$, a density of $\frac{p-1}{p}$, $\frac{p-1}{p^2}$, and
  $\frac{p-1}{p^3}$ satisfy $v_p(\Delta(E))=1$, $v_p(\Delta(E))=2$,
  and $v_p(\Delta(E))=3$, respectively. From this, it is easy to
  compute the density of elliptic curves having prescribed values of
  $\alpha_p$. We may thus compute the densities of $F_3$ and $F_4$ using Proposition \ref{thcountellip}, and
  verify that they are as claimed by the proposition. Finally, as
  before, we may replace $F_3$ and $F_4$ by $F_3\cup \{E:E_{-1}\in
  F_3\}$ and $F_4\cup\{E:E_{-1}\in F_4\}$, respectively, to ensure that
  the last condition is also satisfied.
\end{proof}

The sets $F_{1}$ and $F_{2}$ are
disjoint and defined via congruence conditions modulo powers of $2$
and $3$, while the sets $F_{3}$ and
$F_{4}$ are disjoint and defined via
congruence conditions modulo powers of primes greater than $3$. All
four of these sets are finite unions of large families. Define
\begin{equation}
  \begin{array}{rcl}
    F^+&:=&((F_{1}\cap F_{4})\cup (F_{2}\cap F_{3}))\cap \{E:\Delta(E)>0\},\\[.02in]
    F^-&:=&((F_{1}\cap F_{3})\cup (F_{2}\cap F_{4}))\cap \{E:\Delta(E)<0\}.
  \end{array}
\end{equation}
These sets $F^+$ and $F^-$ are also finite unions of large families
and every elliptic curve $E$ in either of them satisfies $d(E)=-1$ by
construction. We may compute their densities by Propositions
\ref{thcountellip}, \ref{prop23sum}, and \ref{prop5sum} to be at least
$40.914\%$ and $58.534\%$, respectively. This yields Theorem
\ref{fam}.

\section{The average rank of elliptic curves is less than 1}
In this section we prove Theorems \ref{thavgrank}, \ref{rank01}, and
\ref{rank0}. First note that using only Theorem \ref{ellipall} we
obtain the following result:
\begin{proposition}\label{avgbound1}
Let $F$ be a large family of elliptic curves. Then, when elements in
$F$ are ordered by height, we have:
\begin{itemize}
\item[{\rm (a)}] The average $5$-Selmer rank of elliptic curves in $F$
  is bounded by $1.05$.
\item[{\rm (b)}] The set of elliptic curves in $F$ with $5$-Selmer
  rank $0$ or $1$ has density at least $19/24$.
\end{itemize}
\end{proposition}
\begin{proof}
As in the introduction, we note that $20r-15\leq 5^r$ for nonnegative
integers $r$. Therefore, by Theorem \ref{ellipall}, the limsup $\bar{r}_5$ of the
average $5$-Selmer rank of elliptic curves, when ordered
by height, satisfies $20\bar{r}_5-15\leq 6$, proving (a). This bound
is achieved when $95\%$ of elliptic curves have $5$-Selmer rank $1$
and $5\%$ have $5$-Selmer rank $2$.

Let $x_{0{\rm \,or\,}1}$ be the lower density of elliptic curves in $F$ having
$5$-Selmer rank $0$ or $1$. Then, from Theorem \ref{ellipall}, we obtain
$$
x_{0{\rm \,or\,}1}+25(1-x_{0{\rm \,or\,}1})\leq 6,
$$ and hence $x_{0{\rm \,or\,}1}\geq 19/24$,
proving (b). The bound is achieved when a proportion
of $19/24$ of elliptic curves in $F$ have $5$-Selmer rank $0$, and a
proportion of $5/24$ have $5$-Selmer rank $2$.
\end{proof}

We now improve Proposition \ref{avgbound1} in the case of 
large families $F$ having equidistributed root number. 
Recall that the analytic rank of an
elliptic curve $E$ is defined to be the order of vanishing at $1/2$ of
its $L$-function $L(E,s)$.
The evenness or oddness of the analytic rank of $E$ is determined by
whether the sign of the functional equation of $L(E,s)$---the root
number $r(E)$ of $E$---is $1$ or~$-1$, respectively.

The following remarkable result of Dokchitser and Dokchitser \cite{DD}
asserts that the parity of the $p$-Selmer rank of an elliptic curve is
determined by its root number:
\begin{theorem}[Dokchitser--Dokchitser \cite{DD}]\label{thDD}
Let $E$ be an elliptic curve over $\Q$ and let $p$ be any prime.  Let
$s_p(E)$ and $t_p(E)$ denote the rank of the $p$-Selmer group of $E$
and the rank of $E(\Q)[p]$, respectively.  Then the quantity
$s_p(E)-t_p(E)$ is even if and only if the root number of $E$ is~$+1$.
\end{theorem}

It is widely believed that when elliptic curves are ordered by height,
$50\%$ have root number $1$ and $50\%$ have root number $-1$. The same
is believed true in any large family as well. For such families
we have the following result whose proof is similar to \cite[Theorem~39]{TC}.
\begin{proposition}\label{avgbound2}
Let $F$ be a large family of elliptic curves such that exactly
$50\%$ of the curves in $F$, when ordered by height, have root number
$1$. Then we have:
\begin{itemize}
\item[{\rm (a)}] The average $5$-Selmer rank of elliptic curves in $F$
  is bounded above by $.75$.
 \item[{\rm (b)}] The set of elliptic curves in $F$ with
   $5$-Selmer rank $0$ or $1$ has density at least $7/8$.
\item[{\rm (c)}] The set of elliptic curves in $F$ with $5$-Selmer
  rank $0$ has density at least $3/8$.
\end{itemize}
\end{proposition}
\begin{proof}
Note that $12n+1\leq 5^n$ for $n$ even and $60n-55\leq 5^n$ for $n$
odd. Let $\bar{r}_5^{{\rm even}}$  (resp.\ $\bar{r}_5^{{\rm odd}}$) denote
the lim sup of the average $5$-Selmer rank of elliptic curves having even (resp.\ odd)
$5$-Selmer rank. Since the root number in $F$ is equidistributed, Theorems~\ref{ellipall} and \ref{thDD} imply that
\begin{equation}
6\bar{r}_5^{{\rm even}}+30\bar{r}_5^{{\rm odd}}\leq 6-\frac12+\frac{55}2=33.
\end{equation}
Under the above constraint, $\bar{r}_5^{{\rm even}}+\bar{r}_5^{{\rm
    odd}}$ is clearly maximized when $\bar{r}_5^{{\rm odd}}$ is
minimized, which happens when $\bar{r}_5^{{\rm odd}}=1$. Therefore, we
have $(\bar{r}_5^{{\rm even}}+\bar{r}_5^{{\rm odd}})/2\leq
(1/2+1)/2=.75$, proving (a). This bound is achieved when $37.5\%$ of
elliptic curves have $5$-Selmer rank $0$, $50\%$ of elliptic curves
have $5$-Selmer rank $1$, and $12.5\%$ of elliptic curves have
$5$-Selmer rank $2$.

Let $x_{0{\rm \,or\,}1}$ denote the lower density of elliptic curves with 5-Selmer rank 0 or 1.  
Also, let $x_0$ (resp.\ $x_1$) denote the lower density of elliptic curves
with $5$-Selmer rank $0$ (resp.\ $1$).  By Theorems~\ref{ellipall} and~\ref{thDD}, we have
$$ x_0+25(1/2-x_0)+5(x_1+25(1/2-x_1))\leq 6.$$ Thus, we obtain
$24x_0+120x_1\geq 69$. In conjunction with the constraint $x_1\leq 1/2$, it
follows that $x_{0{\rm \,or\,}1}\geq x_0+x_1\geq 7/8$, proving (b). Again, this bound is
achieved when $37.5\%$ of elliptic curves have $5$-Selmer rank $0$,
$50\%$ of elliptic curves have $5$-Selmer rank $1$, and $12.5\%$ of
elliptic curves have $5$-Selmer rank $2$.

Finally, let $x_0$ again denote the lower density of elliptic curves with
$5$-Selmer rank $0$. Since $50\%$ of elliptic curves in $F$ have odd
$5$-Selmer rank, we obtain from Theorems \ref{ellipall} and
\ref{thDD}, that
$$ x_0+25(1/2-x_0)+5/2\leq 6.
$$ It follows that $x_0+25(1/2-x_0)\leq 7/2$ and thus $x_0\geq 3/8$,
proving (c). Once again, this bound is achieved when $37.5\%$ of
elliptic curves have $5$-Selmer rank $0$, $50\%$ of elliptic curves
have $5$-Selmer rank $1$, and $12.5\%$ of elliptic curves have
$5$-Selmer rank $2$.
\end{proof}

Theorems \ref{thavgrank}, \ref{rank01}, and \ref{rank0} now follow by
applying Proposition~\ref{avgbound2} on the family $F$ constructed in
Theorem~\ref{fam}, applying Proposition~\ref{avgbound1} on the
complement of $F$, and noting that the $5$-Selmer rank of an elliptic
curve $E$ is an upper bound for its rank $r(E)$. Since the density of
$F$ is $\mu(F)\geq .5501$, we have by Part (a) of Propositions
\ref{avgbound2} and \ref{avgbound1} that the average rank of elliptic
curves is at most
$$.5501\times.75+.4499\times1.05=.88497<.885$$ which proves Theorem
\ref{thavgrank}. By Part (b) of the propositions, we see that the
lower density of elliptic curves with rank $0$ or $1$ is at least
$$.5501\times7/8+.4499\times19/24\geq .8375$$ which proves Theorem
\ref{rank01}. Finally, Part (c) of Proposition \ref{avgbound2} implies
that the lower density of elliptic curves with rank $0$ is at
least $$.5501\times3/8\geq .2062$$ which proves Theorem
\ref{rank0}.

\subsection*{Acknowledgments}
We thank Noam Elkies, Tom Fisher, Dorian Goldfeld, Dick Gross, Roger
Heath-Brown, Harald Helfgott, Wei Ho, Henryk Iwaniec, Barry Mazur,
Bjorn Poonen, Peter Sarnak, Michael Stoll, Alice Silverberg, Christopher Skinner, Jacob Tsimerman,
Jerry Wang, and Kevin Wilson for helpful conversations. The first
author was partially supported
by NSF Grant~DMS-1001828.


\begin{thebibliography}{10}

\bibitem{MMM} S.\ Y.\ An, S.\ Y.\ Kim, D.\ C.\ Marshall, S.\ H.\
  Marshall, W.\ G.\ McCallum, and A.\ R.\ Perlis, Jacobians of genus one
  curves. {\it J.\ Number Theory} {\bf 90} (2001), no.\ 2, 304--315.

\bibitem{ATR}
M.\ Artin, F.\ Rodriguez-Villegas, and J.\ Tate, On the Jacobians of plane cubics,
{\it Adv. Math.} {\bf 198}  (2005), no.\ 1, 366--382.

\bibitem{BMSW} B.\ Bektemirov, B.\ Mazur, W.\ Stein, M.\ Watkins,
  Average ranks of elliptic curves: tension between data and
  conjecture, {\it Bull.\ Amer.\ Math.\ Soc.\ $($N.S.$)$} {\bf 44} (2007),
  no.\ 2, 233--254 (electronic).

\bibitem{dodqf}
M.\ Bhargava, The density of discriminants of quartic rings and
fields, {\it Ann. of Math.} {\bf 162} (2005), 1031--1063.

\bibitem{dodpf}
M.\ Bhargava, The density of discriminants of quintic rings and
fields, {\it Ann.\ of Math.} {\bf 172} (2010), 1559--1591.

\bibitem{geosieve} 
M.\ Bhargava, The geometric sieve and squarefree values of 
polynomial discriminants and other invariant polynomials,  
preprint.


\bibitem{BhHo}
M.\ Bhargava and W.\ Ho, Coregular representations and genus one
curves, {\tt  http://arxiv.org/abs/1306.4424v1}.

\bibitem{BS}
M.\ Bhargava and A.\ Shankar, Binary quartic forms having bounded
invariants, and the boundedness of the average rank of elliptic
curves, {\tt http://arxiv.org/abs/1006.1002}, {\it Ann.\ of Math.}, to
appear. 

\bibitem{TC} M.\ Bhargava and A.\ Shankar, Ternary cubic forms having
  bounded invariants and the existence of a positive proportion of
  elliptic curves having rank 0, {\tt http://arxiv.org/abs/1007.0052}, {\it Ann.\ of Math.}, to
appear. 

\bibitem{foursel} M.\ Bhargava and A.\ Shankar, The average number of
  elements in the $4$-Selmer groups of elliptic curves is $7$,
{\tt http://arxiv.org/abs/1312.7333}\,.

\bibitem{BSD}
 B.\ J.\ Birch and H.\ P.\ F.\ Swinnerton-Dyer, Notes on elliptic curves. I.  
{\it J. Reine Angew. Math.} {\bf 212} 1963 7--25.

\bibitem{BH}
A.\ Borel and Harish-Chandra, 
Arithmetic subgroups of algebraic groups,
{\it Ann.\ of Math.} {\bf 75} (1962), 485--535. 

\bibitem{Brumer}
A.\ Brumer, The average rank of elliptic curves. I.  {\it Invent. Math.} {\bf 109} (1992), no. 3, 445--472.

\bibitem{BK}
A.\ Brumer and K.\ Kramer, The rank of elliptic curves, {\it Duke Math J.}
{\bf 44}  (1977), no.\ 4, 715--743.

\bibitem{BE1} D.\ A.\ Buchsbaum and D.\ Eisenbud, Gorenstein ideals of height 3. {\it Seminar D. Eisenbud/B. Singh/W. Vogel, Vol. 2}, 30–48, Teubner-Texte zur Math., {\bf 48}, {\it Teubner, Leipzig}, 1982.

\bibitem{BE2} D.\ A.\ Buchsbaum and D.\ Eisenbud, Algebra structures for finite free resolutions, and some structure theorems for ideals of codimension 3, {\it Amer. J. Math.} {\bf 99} (1977), no. 3, 447--485.

\bibitem{Cassels}
J.\ W.\ S.\ Cassels, Arithmetic on curves of genus $1$, IV.\ Proof of the
Hauptvermutung, {\it J. Reine Angew. Math.} {\bf 211} (1962), 95--112.

\bibitem{algebra} J.\ Cremona, T.\ Fisher, C.\ O'Neil, D.\ Simon, and
  M.\ Stoll, Explicit $n$-descent on elliptic curves.\ I.\ Algebra,
{\it  J. Reine Angew. Math.} {\bf 615} (2008), 121--155. 

\bibitem{CFS} J.\ Cremona, T.\ Fisher, and M.\ Stoll, Minimisation and
  reduction of $2$-, $3$- and $4$-coverings of elliptic curves, {\it Journal of 
Algebra Number Theory} {\bf 4} (2010), no.\ 6, 763--820.

\bibitem{Davenport1}
H.\ Davenport, On a principle of Lipschitz,
{\it J.\ London Math.\ Soc.} {\bf 26} (1951), 179--183.  
Corrigendum: ``On a principle of Lipschitz'',  {\it J.\ London Math.\
  Soc.} {\bf 39} (1964), 580. 
 
\bibitem{Davenport2}
H.\ Davenport, On the class-number of binary cubic forms I and II,
{\it J.\ London Math.\ Soc.} {\bf 26} (1951), 183--198.
 
\bibitem{DH} 
H.\ Davenport and H.\ Heilbronn, On the density of discriminants of
cubic fields II, {\it Proc.\ Roy.\ Soc.\ London Ser.\ A} {\bf 322} (1971), 
no.\ 1551, 405--420.

\bibitem{Deu} M.\ Deuring, \,{\it Die Typen der Multiplikatorenringe
  elliptischer Funktionenk\"orper}, {\it Abh.\ Math.\ Sem.\ Hansischen
  Univ.} {\bf 14} (1941), 197--272.

\bibitem{DD}
T.\ Dokchitser and V.\ Dokchitser, On the Birch--Swinnerton-Dyer
quotients modulo squares, {\it Ann.\ of Math.} {\bf 172} (2010), no.\
1, 567--596.

\bibitem{Fisher1}
T.\ Fisher, The invariants of a genus one curve, {\it Proc. Lond. Math. Soc. (3)} {\bf 97} (2008), 753--782.

\bibitem{Fisher3}
T.\ Fisher, Invariant theory for the elliptic normal quintic I. Twists of X(5), {\it Math.\ Ann.} {\bf 356} (2013), no.\ 2, 589--616.


\bibitem{Fisher15}
T.\ Fisher, Invariant theory for the elliptic normal quintic, II. The covering
map, preprint.

\bibitem{Fishermin}
T.\ Fisher, Minimization and reduction of $5$-coverings of elliptic
curves, {\it Algebra \& Number Theory}, to appear.

\bibitem{Fisher2}
T.\ Fisher, On genus one models of degree $5$ with square-free discriminant, preprint.

\bibitem{G1} D.\ Goldfeld, Conjectures on elliptic curves over
  quadratic fields, {\it Number theory, Carbondale 1979 (Proc.
    Southern Illinois Conf., Southern Illinois Univ., Carbondale,
    Ill., 1979)}, pp. 108--118, Lecture Notes in Math., {\bf 751},
  Springer, Berlin, 1979.

\bibitem{Hal} E.\ Halberstadt, Signes locaux des courbes elliptiques
  en 2 et 3, {\it C.\ R.\ Acad.\ Sci.\ Paris S\'er.\ I Math.} {\bf 326}
  (1998), no.\ 9, 1047--1052.

\bibitem{HB}
D.\ R.\ Heath-Brown, The average analytic rank of elliptic curves. {\it Duke Math. J.}
{\bf 122} (2004), no. 3, 591--623.

\bibitem{Helfgott} H.\ A.\ Helfgott, Root numbers and the parity
  problem, Ph.D. Thesis, Princeton University, June 2003.

\bibitem{KS} N.\ M.\ Katz and P.\ Sarnak, {\it Random matrices,
    Frobenius eigenvalues, and monodromy}, American Mathematical
  Society Colloquium Publications {\bf 45}, American Mathematical
  Society, Providence, RI, 1999.

\bibitem{Kr}
A.\ Kraus, Quelques remarques {\`a }Êpropos des invariants 
$c_4,\;c_6$ et $\Delta$ d'une courbe elliptique,
{\it  Acta Arith.} {\bf 54}  (1989), no.\ 1, 75--80.

\bibitem{littelmann} P.\ Littelmann, Koregul\"are und \"aquidimensionale {D}arstellungen, {\it Journal of Algebra}, {\bf 123}, (1989), 193--222.

\bibitem{PR}
B.\ Poonen and E.\ Rains, Random maximal isotropic subspaces and Selmer groups,
{\it J.\ Amer.\ Math.\ Soc.} {\bf 25} (2012), 245--269.

\bibitem{popovvinberg}
V.\ L.\ Popov and E.\ B.\ Vinberg, Invariant Theory, in {\it Algebraic Geometry IV}, Encylopaedia of Mathematical Sciences {\bf 55}, Springer-Verlag, 1994.

\bibitem{Rh}
D.\ E.\ Rohrlich, Variation of the root number in families of elliptic curves,  
{\it Compositio Math.} {\bf 87} (1993), no.\ 2, 119--151.

\bibitem{Sil}
J.\ H.\ Silverman, The arithmetic of elliptic curves, GTM {\bf 106}, Springer-Verlag, 1986.

\bibitem{SW} S.\ Wong, On the density of elliptic curves, {\it
  Compositio Math.} {\bf 127} (2001), no.\ 1, 23--54.


\bibitem{Young} M.\ P.\ Young, Low-lying zeros of families of elliptic
  curves.  {\it J. Amer.\ Math.\ Soc.},  {\bf 19} (2006), no.~1,
  205--250.


\end{thebibliography}
\end{document}